\documentclass[11pt,letterpaper]{amsart}
\usepackage{amsfonts}
\usepackage{amsmath}
\usepackage{amsthm}
\usepackage{amssymb}
\usepackage{amsmath,amscd}
\usepackage{subcaption}
\usepackage[pdffitwindow=false,
colorlinks=true,linkcolor=black,urlcolor=black,citecolor=black]{hyperref}
\usepackage{tabularx}
\usepackage{float}
\usepackage{multirow}

\newtheorem{theorem}{Theorem}[section]
\newtheorem{lemma}[theorem]{Lemma}
\theoremstyle{definition}

\newtheorem{example}[theorem]{Example}
\newtheorem{proposition}[theorem]{Proposition}

\newtheorem{remark}[theorem]{Remark}
\numberwithin{equation}{section}
\newcommand{\spin}{\mathrm{Spin}}
\newcommand{\trace}{\mathrm{tr}}
\newcommand{\ricci}{\mathrm{Ric}}

\title{Existence of cohomogeneity one Einstein metrics}
\author{Hanci Chi}
\address{Department of Pure Mathematics\\ Xi'an Jiaotong-Liverpool University\\ Suzhou 215123\\ China}
\email{hanci.chi@xjtlu.edu.cn}
\begin{document}
\maketitle
\begin{abstract}
This paper derives a sufficient condition for the existence of cohomogeneity one Einstein metrics on double disk bundles of two summands type. The condition is an inequality that involves geometric data from the principal orbits. 
\end{abstract}
\let\thefootnote\relax\footnote{2020 Mathematics Subject Classification: 53C25 (primary).

Keywords: Einstein metric, cohomogeneity one metric.

The author is supported by the NSFC (No. 12071489, No. 12301078), the Natural Science Foundation of Jiangsu Province (BK-20220282), and the XJTLU Research Development Funding (RDF-21-02-083).}

%\tableofcontents

\section{Introduction}
A Riemannian metric $g$ is Einstein if $\ricci(g)=\Lambda g$ for some constant $\Lambda$. The search for Einstein metrics continues to be a challenging problem. According to the sign of $\Lambda$, one speaks of positive ($\Lambda>0$), Ricci-flat ($\Lambda=0$), or negative ($\Lambda<0$) Einstein metrics.

A milestone in this field is the existence theory of Kähler--Einstein metrics, with fundamental results including the resolution of the Calabi conjecture in the non-positive case \cite{aubin1976equations, yau1978ricci} and the Yau--Tian--Donaldson conjecture for Fano manifolds \cite{chen2015kahler,tian2015k}. While these results have profoundly shaped our understanding of Einstein metrics, many challenges remain in the non-Kähler setting. Significant efforts have been devoted to constructing Einstein metrics in the real Riemannian context. One fruitful approach involves considering homogeneous spaces, where the high degree of symmetry simplifies the Einstein equation to an algebraic system. In the case of homogeneous spaces, a general existence theory has been developed in \cite{wang1986existence, bohm2004variational, bohm2004homogeneous}. It is shown in \cite{bohm2006low} that each compact simply connected homogeneous space of dimension less than $12$ admits a homogeneous Einstein metric. For low dimensional homogeneous Einstein manifolds up to dimension $7$, see \cite{jensen1969homogeneous, gorbatsevich1977classification, alekseevsky1996homogeneous, nikonorov1999six, nikonorov2003compact, nikonorov2004compact}. 
 
To explore Einstein metrics with large symmetry but beyond homogeneity, a natural direction is cohomogeneity one Einstein metrics, where the isometry group acts with principal orbits $\mathsf{G}/\mathsf{K}$ of codimension one. The cohomogeneity one condition transforms the Einstein equation to a system of ODEs. A landmark example is Page’s metric on $\mathbb{CP}^2 \sharp\overline{\mathbb{CP}}^{2}$, the first known inhomogeneous positive Einstein metric \cite{page_compact_1978}. Page's metric was generalized in a series of works including \cite{berard-bergery_sur_1982, koiso1986non, page1987inhomogeneous, koiso_non-homogeneous_1988, wang_einstein_1998}. In \cite{bohm_inhomogeneous_1998}, new cohomogeneity one Einstein metrics were found on some low dimensional manifolds, including spheres and $\mathbb{HP}^2 \sharp\overline{\mathbb{HP}}^{2}$. The Einstein metric on the latter space can be viewed as a quaternionic counterpart of Page's metric. B\"ohm's metric is generalized to all high dimensional $\mathbb{HP}^{m+1} \sharp\overline{\mathbb{HP}}^{m+1}$ in \cite{chi2024positive}, confirming numerical evidence in \cite{page1986einstein}. 

This paper further extends the existence result in \cite{chi2024positive} to a broader class of closed cohomogeneity one manifolds.  Specifically, the principal orbit is the total space of a homogeneous fibration
\begin{equation}
\label{eqn_homogeneous fibration}
\mathbb{S}^{d_1}=\mathsf{H}/\mathsf{K}\hookrightarrow \mathsf{G}/\mathsf{K}\rightarrow \mathsf{G}/\mathsf{H}=\mathcal{Q}^{d_2}.
\end{equation}
The compact cohomogeneity one manifold $M$ of dimension $n+1=d_1+d_2+1$ is an $\mathbb{S}^{d_1+1}$-bundle over $\mathcal{Q}^{d_2}$. Since the manifold $M$ is formed by gluing two copies of a disk bundle over $\mathcal{Q}^{d_2}$ along $\mathsf{G}/\mathsf{K}$, it is also referred to as a \emph{double disk bundle} associated to the group triple $(\mathsf{K},\mathsf{H},\mathsf{G})$. We consider double disk bundles of \emph{two summands type}, where the cohomogeneity one Einstein equation admits an ansatz consisting of two metric components $f_1^2$ and $f_2^2$, respectively for $\mathfrak{h}/\mathfrak{k}$ and $\mathfrak{g}/\mathfrak{h}$. Cases where the isotropy representations of principal orbits split into two irreducible summands are included. The cohomogeneity one Einstein equation on $M$ is characterized by the structural triple $(d_1,d_2,A)$ associated with $\mathsf{G}/\mathsf{K}$, where the constant $A\geq 0$ is the norm square of the fundamental tensor in the theory of Riemannian submersions \cite{o1966fundamental}. With such a geometric setting, we prove the following existence theorem.
\begin{theorem}
\label{thm_main 1}
For any $(d_1,d_2)$ with $d_1,d_2\geq 2$, there exists a constant $\chi_{d_1,d_2}\in \left(0,\frac{d_2(d_2-1)^2}{d_1^2(d_1d_2-d_2+4)}\right]$ such that if $\mathsf{G}/\mathsf{K}$ is a principal orbit with the structural triple $(d_1,d_2,A)$ and $A\in (0,\chi_{d_1,d_2})$, there exists at least one cohomogeneity one Einstein metric on $M$.
\end{theorem}
The constant $\chi_{d_1,d_2}$ is determined by several algebraic functions of $(d_1,d_2)$. The explicit formula of $\chi_{d_1,d_2}$ for a general $(d_1, d_2)$ is fairly complicated. We present some examples in Appendix~\ref{appendix:PEexamples}. In an upcoming paper \cite{chi2025estimate}, we derive some estimates of $\chi_{d_1,d_2}$ to provide more examples of new compact Einstein manifolds.

\begin{remark}
The case $d_1=1$ was already extensively studied in the work of Bérard-Bergery  \cite{berard-bergery_sur_1982} and subsequent works \cite{koiso1986non, page1987inhomogeneous, koiso_non-homogeneous_1988, wang_einstein_1998}, where no homogeneity assumption on the hypersurface was required. Hence, our theorem is stated under the assumption $d_1 \geq 2$. In principle, our result could be generalized to sphere bundles over an inhomogeneous Einstein manifold. However, in such a setting, the connection on the sphere bundle (as the hypersurface) needs to be Yang--Mills, with the pointwise norm of its curvature being constant; see \cite[Theorem 9.73]{besse2007einstein}. How frequently such sphere bundles arise outside the homogeneous fibration context is, to the best of the author's knowledge, unclear.
\end{remark}

\begin{remark}
By Myers' Theorem \cite{myers_riemannian_1941}, a compact manifold admitting a positive Einstein metric must have a finite fundamental group. If $d_2=1$, the base space is necessarily $\mathbb{S}^1$, which implies that the fundamental group $\pi_1(M)$ is not finite. Hence, such a double disk bundle $M$ cannot admit any metric with positive Ricci curvature, and we therefore exclude the case $d_2=1$.
\end{remark}

\begin{remark}
Theorem \ref{thm_main 1} also applies when $A=0$, in which case one simply recovers the product metric on $\mathbb{S}^{d_1+1}\times \mathcal{Q}^{d_2}$. Therefore, Einstein metrics obtained in Theorem \ref{thm_main 1} with $A>0$ may be regarded as analogues of the product Einstein metric on $\mathbb{S}^{d_1+1}\times \mathcal{Q}^{d_2}$. It is known that for $5\leq n\leq 8$, there exist infinitely many cohomogeneity one Einstein metrics on $\mathbb{S}^{d_1+1}\times \mathcal{Q}^{d_2}$ \cite{bohm_inhomogeneous_1998}. Whether additional non-standard Einstein metrics exist on higher dimensional product spaces is beyond the scope of this paper.
\end{remark}

B\"ohm's work \cite{bohm_inhomogeneous_1998} established the existence of infinitely many cohomogeneity one Einstein metrics by exploiting a local spiraling mechanism: when certain critical points on the Ricci–flat subsystem are foci, integral curves can wind arbitrarily many times around the cone solution. An intermediate value argument combined with symmetry then produces infinitely many heteroclines that represent complete Einstein metrics on $\mathbb{S}^{d_1+1}\times \mathcal{Q}^{d_2}$. 

In contrast, for the broader class of examples considered here (those with $A>0$), no examples are known where the relevant critical points remain foci, so the local spiraling argument is unavailable. Hence our approach is necessarily different: by constructing global barriers and capturing a distinguished integral curve, together with an intermediate value argument and symmetry, we prove the existence of at least one heterocline that represents a complete Einstein metric on $M$.

This paper is organized as follows. In Section \ref{sec_Cohomogeneity one Einstein equations}, we introduce the dynamical system for cohomogeneity one Einstein metrics of two summands type. In Section \ref{sec_Local Analysis}, we perform local analysis on critical points representing the initial and terminal conditions. In Section \ref{sec_Global Analysis}, we generalize the global analysis in \cite{chi2024positive} and prove the main theorem. For the sake of readability, formulas and computations obtained by Maple are presented in Appendix~\ref{appendix:computation}.

\section{The cohomogeneity one Einstein equation}
\label{sec_Cohomogeneity one Einstein equations}
For a group triple $(\mathsf{K},\mathsf{H},\mathsf{G})$ as in \eqref{eqn_homogeneous fibration}, let $\mathsf{G}$ be compact and semisimple. Let $\tilde{b}$ be a normal homogeneous metric on $\mathsf{G}/\mathsf{K}$, that is, a metric induced from an $\mathrm{Ad}(\mathsf{G})$-invariant inner product on the Lie algebra $\mathfrak{g}$. If $\mathsf{G}$ is simple, we take $\tilde{b}$ to be the standard homogeneous metric induced by the negative Killing form $-B_\mathfrak{g}$. Assume that there exists a $\tilde{b}$-orthogonal decomposition $\mathfrak{g}=\mathfrak{k}\oplus \mathfrak{h}/\mathfrak{k}\oplus \mathfrak{g}/\mathfrak{h}$, where the latter two summands are $K$-invariant (not necessarily irreducible). Then there exists a 2-parameter family of homogeneous
metrics on $\mathsf{G}/\mathsf{K}$ with the parameters controlling $\mathbb{S}^{d_1}$ and $\mathcal{Q}^{d_2}$.
We further assume that the Ricci endomorphism of $\mathsf{G}/\mathsf{K}$ is diagonal for such a 2-parameter family. We fix a background metric $b$ by rescaling $\tilde{b}$ so that the homogeneous metric
\begin{equation}
\label{eqn_diagonal homogeneous metric}
f_1^2\left.b\right|_{\mathfrak{h}/\mathfrak{k}}+f_2^2\left.b\right|_{\mathfrak{g}/\mathfrak{h}}
\end{equation}
has Ricci endomorphism $\mathrm{diag}(r_1 I_{d_1}, r_2 I_{d_2})$, where
\begin{equation}
\label{eqn_Ricci tensor on G/K}
\begin{split}
%d_1&=\dim(\mathfrak{h}/\mathfrak{k}),\quad d_2=\dim(\mathfrak{g}/\mathfrak{h}),\\
r_1&=\frac{d_1-1}{f_1^2}+A\frac{f_1^2}{f_2^4},\quad r_2=\frac{d_2-1}{f_2^2}-2\frac{d_1}{d_2} A\frac{f_1^2}{f_2^4}.
\end{split}
\end{equation}

%*****
%
%Let $t$ denote the arc-length parameter along a geodesic that meets all principal orbits $\mathsf{G}/\mathsf{K}$ transversely. In this coordinate, the cohomogeneity one Einstein equation \eqref{eqn: Einstein equation} is expressed as a system of second-order ODEs in $t$. Following the approach of B\"ohm \cite{bohm_inhomogeneous_1998}, it is convenient to introduce a new variable
%\[
%d\eta = \sqrt{\trace^2(L)+n\Lambda}\, dt,
%\]
%which effectively rescales the system and allows for a more transparent analysis of the flow. In terms of $\eta$, the initial and terminal conditions at the singular orbits correspond to $\eta \to -\infty$ and $\eta \to +\infty$, respectively. Primes $'$ will denote differentiation with respect to $\eta$, while dots $\dot{}$ previously indicated differentiation with respect to $t$. 
%
%******

With a principal orbit $\mathsf{G}/\mathsf{K}$ as above, each $f_i$ is a function defined on the $1$-dimensional orbit space parametrized by $t$. From \cite{eschenburg_initial_2000}, the shape operator of $\mathsf{G}/\mathsf{K}$ in $M$ is $L:=\mathrm{diag}\left(\frac{\dot{f}_1}{f_1}I_{d_1},\frac{\dot{f}_2}{f_2}I_{d_2}\right)$. The cohomogeneity one Einstein equation is
\begin{align}
\label{eqn: Einstein equation}
\frac{\ddot{f_i}}{f_i}-\left(\frac{\dot{f_i}}{f_i}\right)^2&=-\trace(L)\frac{\dot{f_i}}{f_i}+r_i-\Lambda,\quad i=1,2;\\
\label{eqn: conservation law in Einstein equation}
d_1\frac{\ddot{f_1}}{f_1}+d_2\frac{\ddot{f_2}}{f_2}&=-\Lambda.
\end{align}
Alternatively, \eqref{eqn: conservation law in Einstein equation} can be replaced by
\begin{equation}
\label{eqn_original conservation law}
\trace(L^2)-\trace^2(L)+d_1r_1+d_2r_2-(n-1)\Lambda=0.
\end{equation}
A solution $(f_1,f_2)$ defines a cohomogeneity one Einstein metric
\begin{equation}
\label{eqn_two summand ansatz}
dt^2+f_1^2(t)\left.b\right|_{\mathfrak{h}/\mathfrak{k}}+f_2^2(t)\left.b\right|_{\mathfrak{g}/\mathfrak{h}}
\end{equation}
on $M$ if $\mathsf{G}/\mathsf{K}$ collapses to $\mathsf{G}/\mathsf{H}$ on both ends. By \cite{eschenburg_initial_2000} and more recently \cite{verdiani2022smoothness}, the smooth extension to the singular orbit requires the following initial and terminal conditions.
\begin{equation}
\label{eqn_initial condition}
\lim\limits_{t\to 0}(f_1,\dot{f_1},f_2,\dot{f_2})=(0,1,\underline{f},0),\quad \lim\limits_{t\to t_*}(f_1,\dot{f_1},f_2,\dot{f_2})=(0,-1,\overline{f},0),\quad \underline{f},\overline{f}>0.
\end{equation}

For a complete solution defined on $[0,t_*]$, many geometric quantities, including the mean curvature $\trace{(L)}$, diverge at both ends of the interval. To compactify the phase space of the dynamical system, we introduce a natural parameter $\eta$ by 
$$d\eta=\sqrt{\trace^2(L)+n\Lambda}dt.$$
Because $\trace(L)$ blows up as $t\to 0$ and $t\to t_*$, the interval $[0,t_*]$ is mapped to the entire real line in the $\eta$–coordinate. Under this change of variable, equations \eqref{eqn: Einstein equation}–\eqref{eqn: conservation law in Einstein equation} are transformed into a dynamical system on a compact phase space, which is well suited for the subsequent global analysis. Specifically, define 
$$
X_1:=\frac{\frac{\dot{f_1}}{f_1}}{\sqrt{\trace^2(L)+n\Lambda}},\quad X_2:=\frac{\frac{\dot{f_2}}{f_2}}{\sqrt{\trace^2(L)+n\Lambda}},$$
$$ Y:=\frac{\frac{1}{f_1}}{\sqrt{\trace^2(L)+n\Lambda}},\quad Z:=\frac{\frac{f_1}{f_2^2}}{\sqrt{\trace^2(L)+n\Lambda}}.
$$
Each $X_i$ is the normalized principal curvature, while the metric components are encoded in $Y$ and $Z$.
Define functions 
\begin{equation}
\label{eqn_curvature for G/K}
\begin{split}
&G:=d_1X_1^2+d_2X_2^2,\quad H:=d_1X_1+d_2X_2,\\
&R_1:=(d_1-1)Y^2+AZ^2,\quad R_2:=(d_2-1)YZ-\frac{2d_1}{d_2}AZ^2.
\end{split}
\end{equation}
Up to the normalization factor, the function $G$ represents the squared norm of the shape operator, while $H$ corresponds to the mean curvature. The homogeneous Ricci curvatures of the principal orbit along the fiber and base directions are respectively captured by $R_1$ and $R_2$. 

Here and in what follows, the prime notation ($'$) denotes differentiation with respect to $\eta$. Equations \eqref{eqn: Einstein equation} are transformed to 
\begin{equation}
\label{eqn_New Einstein equation}
\begin{split}
X_i'&=X_iH\left(G+\frac{1-H^2}{n}-1\right)+R_i-\frac{1-H^2}{n},\quad i=1,2,\\
Y'&=Y\left(H\left(G+\frac{1-H^2}{n}\right)-X_1\right),\\
Z'&=Z\left(H\left(G+\frac{1-H^2}{n}\right)+X_1-2X_2\right).\\
\end{split}
\end{equation}
The conservation law \eqref{eqn_original conservation law} is transformed to the following algebraic surface.
\begin{equation}
\label{eqn_conservation equation for 0}
\mathcal{C}:=\left\{\frac{d_1d_2}{n}(X_1-X_2)^2+d_1R_1+d_2R_2=1-\frac{1}{n}\right\}.
\end{equation}
Equivalently, we have 
\begin{equation}
\label{eqn_conservation equation for 1}
\mathcal{C}=\left\{\frac{1}{n-1}\left(G-H^2+d_1R_1+d_2R_2\right)=\frac{1-H^2}{n}\right\}.
\end{equation}

By a direct computation, one verifies that the algebraic surface $\mathcal{C}$ is invariant under the flow of \eqref{eqn_New Einstein equation}, without appealing to \eqref{eqn_original conservation law}. Recall that $Y$ and $Z$ encode metric components, so along the solutions of interest they are naturally non-negative. In the new coordinates, the range of the mean curvature is compressed from the entire real line to the open interval $(-1,1)$. These geometric constraints are reflected in the structure of the dynamical system, as stated below.
\begin{proposition}
The set
$$\mathcal{E}:=\mathcal{C}\cap \{Y,Z\geq 0\}\cap \{H^2\leq 1\}$$
is invariant under the flow of \eqref{eqn_New Einstein equation}.
\end{proposition}
\begin{proof}
Equation \eqref{eqn_New Einstein equation} shows that the sign of $Y$ and $Z$ is preserved, so we may restrict to $Y,Z\geq 0$. Moreover, differentiating $H$ gives
\begin{equation}
\label{eqn_H'}
H'=(H^2-1)\left(G+\frac{1}{n}(1-H^2)\right)=(H^2-1)\left(\tfrac{1}{n}+\tfrac{d_1d_2}{n}(X_1-X_2)^2\right).
\end{equation}
Hence the hypersurfaces $\{H=\pm 1\}$ are flow-invariant. In particular, if $H^2<1$ at some point along an integral curve, the integral curve cannot cross $\{H^2=1\}$, showing that $\mathcal{E}$ is flow-invariant.
\end{proof}

\begin{remark}
\label{rem_Z2}
The region $\{H^2<1\}$ corresponds to positive Einstein metrics ($\Lambda>0$). The system admits a $\mathbb{Z}_2$-symmetry
$$
(\eta,X_1,X_2,Y,Z) \longleftrightarrow (-\eta,-X_1,-X_2,Y,Z),
$$
which corresponds geometrically to reversing the normal direction of the hypersurface. 
This symmetry allows one to reflect an integral curve in $\mathcal{E}\cap \{H>0\}$ that intersects $\{X_1=X_2=0\}$ across $\{H=0\}$ and glue it to its mirror image, thereby producing a complete solution defined on both ends. For such a metric, the initial and terminal conditions of $f_2$ are identical, and we have $\underline{f}=\overline{f}$ in \eqref{eqn_initial condition}.
\end{remark}

\begin{remark}
In the Ricci-flat case ($\Lambda=0$), one may instead consider $d\eta=\trace(L)\,dt$ and redefine the variables $X_1,X_2,Y,Z$ accordingly.  
The resulting dynamical system coincides with the restriction of \eqref{eqn_New Einstein equation} to $\mathcal{RF}^+:=\mathcal{E}\cap \{H=1\}$. By the $\mathbb{Z}_2$-symmetry mentioned above, the subsystem restricted to $\mathcal{RF}^-:=\mathcal{E}\cap \{H=-1\}$ also corresponds to the Ricci-flat case after a time reversal.
\end{remark}

To retrieve the original system, consider 
\begin{equation}
\label{retrieve original coordinate}
t=\int_{-\infty}^\eta \sqrt{\frac{1-H^2}{n\Lambda}}d\tilde{\eta},\quad f_1^2=\frac{1-H^2}{n\Lambda Y^2},\quad f_2^2=\frac{1-H^2}{n\Lambda YZ}.
\end{equation}

\section{Local Analysis}
\label{sec_Local Analysis}
We perform local analysis at some important critical points of \eqref{eqn_New Einstein equation}. The initial and terminal conditions \eqref{eqn_initial condition} are respectively transformed to critical points  
\begin{equation}
\label{eqn_p0}
p_0^\pm=\left(\pm\frac{1}{d_1},0,\frac{1}{d_1},0\right).
\end{equation}
The local existence of cohomogeneity one Einstein metrics is obtained from the local analysis at $p_0^+$. The linearization at $p_0^+$ is 
$$
\begin{bmatrix}
\frac{4d_2-2d_1(d_2-1)}{d_1n} & \frac{d_2(d_1 - 1)(d_1 - d_2)}{d_1^2n}& \frac{2(d_1-1)}{d_1}&0\\
\frac{2d_1}{n}&\frac{1}{d_1}+\frac{d_2-d_1}{n}&0&\frac{d_2-1}{d_1}\\
\frac{2d_2}{d_1n}&\frac{d_2(d_2-d_1)}{d_1^2n}&0&0\\
0&0&0&\frac{2}{d_1}
\end{bmatrix}.
$$
The critical point is hyperbolic. The following are the only two unstable eigenvectors that are tangent to $\mathcal{E}$. Both eigenvectors have the same eigenvalue $\frac{2}{d_1}$,
$$v_1=\begin{bmatrix}
-2d_2(d_2-1)\\
2d_1(d_2-1)\\
-d_2(d_2-1)\\
2d_1(d_1+1)
\end{bmatrix},\quad 
v_2=\begin{bmatrix}
-d_1^2+d_1d_2-n\\
-2d_1^2\\
-d_2\\
0
\end{bmatrix}.
$$

By the Hartman--Grobman theorem, there is a 1-to-1 correspondence between each choice of linearized solution and an actual solution curve that emanates from $p_0^+$. From the unstable version of \cite[Theorem 4.5, Chapter 13]{coddington_theory_1955}, the error of such a correspondence is $O\left(e^{\left(\frac{2}{d_1}+\delta\right)\eta}\right)$. There exists a continuous $1$-parameter family $\gamma_s$ that emanates from $p_0^+$ such that 
\begin{equation}
\label{eqn_linearized solution near p0+}
\gamma_s\sim p_0^+ + (sv_1+v_2)e^{\frac{2}{d_1}\eta}+O\left(e^{\left(\frac{2}{d_1}+\delta\right)\eta}\right).
\end{equation}
As $v_1$ and $v_2$ are tangent to $\mathcal{E}$, each $\gamma_s$ is in $\mathcal{E}$.

By \eqref{retrieve original coordinate} and the linearized solution, we have 
$$\underline{f}^2=\lim\limits_{t\to 0}f_2^2=\lim\limits_{\eta\to-\infty}\frac{1-H^2}{n\Lambda YZ}=\frac{d_1}{\Lambda s}$$
along the integral curve $\gamma_s$.
Eigenvectors $v_1$ and $v_2$ separately correspond to two extreme situations. As $v_1$ is tangent to $\mathcal{RF}^+$, the integral curve that emanates from $p_0^+$ along $v_1$ represents a locally defined Ricci-flat metric around $\mathcal{Q}^{d_2}$. The vector $v_2$ tangent to $\{Z=0\}$ and it leads to an integral curve lying in $\{Z=0\}$, where the metric component along $\mathcal{Q}^{d_2}$ diverges. We restrict to directions of the form $sv_1+v_2$ rather than the more general combination $s_1 v_1+s_2 v_2$ since scalar multiplication does not yield genuinely new metrics: it simply shifts the integral curve along the $\eta$-axis. In the original coordinate, the operation corresponds to a homothetic change of Ricci-flat metrics, or translation in the $t$-axis of positive Einstein metrics.

By \eqref{eqn_linearized solution near p0+}, we have 
\begin{equation}
\label{eqn_initial Z and H}
Z(\gamma_s)\sim 2d_1(d_1+1) e^{\frac{2}{d_1}\eta}+O\left(e^{\left(\frac{2}{d_1}+\delta\right)\eta}\right),\quad  H(\gamma_s)\sim 1-d_1n(d_1+1) e^{\frac{2}{d_1}\eta}+O\left(e^{\left(\frac{2}{d_1}+\delta\right)\eta}\right)
\end{equation}
near $p_0^+$. Therefore, the essential 1-parameter family of non-degenerate locally defined positive Einstein metrics around $\mathcal{Q}^{d_2}$ is represented by $\{\gamma_s\mid s>0\}$. Since the $Y$-coordinate of $p_0^+$ is positive, each $\gamma_s$ with $s>0$ is in the interior 
$$\mathcal{E}^\circ:=\mathcal{E}\cap \{Y,Z>0\}\cap \{H^2<1\}.$$
With \eqref{eqn_linearized solution near p0+}, it is also natural to let $\gamma_\infty$ denote the integral curve that emanates from $p_0^+$ along $v_1$ and stays in $\mathcal{RF}^+$.

To obtain a complete Einstein metric on $M$, one needs to show the existence of an integral curve that joins $p_0^\pm$. In B\"ohm’s analysis for $A=0$, such heteroclines arise because integral curves emanating from $p_0^+$ spiral around the sine cone based over a homogeneous Einstein metric on $\mathsf{G}/\mathsf{K}$. In our setting with $A>0$, a similar mechanism requires the existence of homogeneous Einstein metrics on $\mathsf{G}/\mathsf{K}$. Algebraically, this corresponds to critical points that are characterized by the condition $X_1=X_2=\frac{1}{n}$ and $R_1=R_2$, which reduces to the following quadratic equation,
\begin{equation}
\label{eqn_homogeneous metric polynomial}
\frac{1}{Y^2}(R_1-R_2)=\frac{n+d_1}{d_2}A l^2-(d_2-1)l+(d_1-1)=0,\quad l=\frac{Z}{Y}.
\end{equation}

Let $\Delta$ be the discriminant of \eqref{eqn_homogeneous metric polynomial}. If $\Delta> 0$, or equivalently,
\begin{equation}
\label{eqn_bohm's upper bound for A}
A< \frac{1}{n+d_1}\frac{d_2(d_2-1)^2}{4(d_1-1)},
\end{equation}
there are two real roots $\mu_1>\mu_2$ of the quadratic equation. These roots correspond to metric ratios $\frac{f_1^2}{f_2^2}$ that yield two distinct homogeneous Einstein metrics on $\mathsf{G}/\mathsf{K}$, represented by the following critical points in $\mathcal{RF}^\pm$,
$$
q_i^\pm=\left(\pm\frac{1}{n},\pm\frac{1}{n},y_i,z_i\right),\quad \frac{z_i}{y_i}=\mu_i,\quad R_1(y_i,z_i)=R_2(y_i,z_i)=\frac{n-1}{n^2},\quad i=1,2.
$$
Being treated as trivial integral curves on $\mathcal{RF}^+$, the points $q_i^+$ also represent the Ricci-flat cones over the two homogeneous Einstein metrics, while $q_i^-$ are their time reversal counterparts related by the $\mathbb{Z}_2$-symmetry. 

As $\mu_1$ and $\mu_2$ are roots for \eqref{eqn_homogeneous metric polynomial}, we immediately obtain the following identities:
\begin{equation}
\label{eqn_R1-R2}
R_1-R_2=\frac{n+d_1}{d_2}A(Z-\mu_1 Y)(Z-\mu_2 Y),
\end{equation}
where
\begin{equation}
\label{eqn_basic inequalities for mu_i}
\mu_1+\mu_2=\frac{d_2(d_2-1)}{(n+d_1)A},\quad \mu_1\mu_2=\frac{d_2(d_1-1)}{(n+d_1)A}.
\end{equation}

The two sine cone solutions of the cohomogeneity one Einstein equation that joins $q_i^\pm$ are identified as the invariant sets
$$
\Phi_i:=\mathcal{E}\cap \{X_1-X_2=0\}\cap \{\mu_iY-Z=0\},\quad i=1,2.
$$
A local analysis shows that the cone solution $\Phi_2$ is a focal attractor if
\begin{equation}
\label{eqn_bohm's inequality of existence}
n\leq 8,\quad A<\frac{(9-n)(d_2 n+ 7n + 9d_1)}{(d_1 n-8n-9 d_1)^2}\frac{d_2(d_2 - 1)^2}{4(d_1-1)},
\end{equation}
so there exist integral curves emanating from $p_0^+$ that spiral around the cone with arbitrarily large winding number, which by the $\mathbb{Z}_2$-symmetry leads to infinitely many heteroclines that join $p_0^\pm$. Hence there exist infinitely many cohomogeneity one Einstein metrics on $M$; see \cite[Corollary 5.8]{bohm_inhomogeneous_1998}. Unfortunately, there is no known example of a $\mathsf{G}/\mathsf{K}$ that satisfies \eqref{eqn_bohm's inequality of existence} with $A>0$. We are hence motivated to find a new upper bound for $A$, presumably in the interval $\left(\frac{(9-n)(d_2 n+ 7n + 9d_1)}{(d_1 n-8n-9 d_1)^2}\frac{d_2(d_2 - 1)^2}{4(d_1-1)},\frac{1}{n+d_1}\frac{d_2(d_2-1)^2}{4(d_1-1)}\right)$, for the existence theorem. 

\begin{remark}
\label{rem_non_existence}
If $A\geq \frac{1}{n+d_1}\frac{d_2(d_2-1)^2}{4(d_1-1)}$, there exists at most one $\mathsf{G}$-invariant Einstein metric on $\mathsf{G}/\mathsf{K}$. It is further deduced that the two summands type ansatz does not yield any cohomogeneity one Einstein metric on $M$
if such an inequality holds; see \cite[Theorem 3.1]{bohm_inhomogeneous_1998}. In an upcoming paper \cite{chi2025nonexistence}, we show that for some principal orbit, its structural triple satisfies \eqref{eqn_bohm's upper bound for A} yet does not yield any Einstein metric on $M$ from the two summands type ansatz.
\end{remark}

\section{Global Analysis}
\label{sec_Global Analysis}
As explained in the Introduction, B\"ohm's local spiraling mechanism (which yields infinitely many solutions when $q_2^\pm$ are foci) is generally not available if $A>0$. In this section, we implement a global analysis outlined below: construct suitable barrier functions, single out a distinguished integral curve $\gamma_{s_\bullet}$, and use an intermediate value argument to produce a heterocline joining $p_0^\pm$. We begin by introducing the compact region $\mathcal{E}^+$ and establishing its basic properties. We then define the barrier functions $S,T,P$ and analyze their evolution along $\gamma_s$.

Consider the following set whose boundary contains $p_0^+$ and $\Phi_i$.
$$\mathcal{E}^+:=\mathcal{E}\cap \{H\geq 0\}\cap \{X_1-X_2\geq  0\}\cap \{\mu_1 Y-Z\geq 0\}.$$
\begin{proposition}
\label{prop_S_is compact}
The set $\mathcal{E}^+$ is compact.
\end{proposition}
\begin{proof}
By $\mu_1Y-Z\geq 0$, we have 
\begin{equation}
\label{eqn_scalr curvatrue}
\begin{split}
d_1R_1+d_2R_2&=d_1(d_1-1)Y^2+d_2(d_2-1)YZ-d_1AZ^2\\
&\geq d_1(d_1-1)Y^2+\frac{d_2(d_2-1)-d_1A\mu_1}{\mu_1}Z^2.
\end{split}
\end{equation}
By \eqref{eqn_basic inequalities for mu_i}, we have $A= \frac{d_2(d_2-1)}{(n+d_1)(\mu_1+\mu_2)}\leq \frac{d_2(d_2-1)}{d_1\mu_1}$. The coefficient of $Z^2$ in \eqref{eqn_scalr curvatrue} is positive. Therefore, the functions $X_1-X_2$, $Y$ and $Z$ are bounded by \eqref{eqn_conservation equation for 0}. As $H^2\leq 1$ in $\mathcal{E}$, the functions $X_1$ and $X_2$ are bounded in a compact parallelogram region. The proof is complete.
\end{proof}

\begin{proposition}
\label{prop_each orbit transversely escapes}
Every integral curve in the interior 
$$(\mathcal{E}^+)^\circ:=\mathcal{E}^\circ \cap \{H> 0\}\cap \{X_1-X_2> 0\}\cap \{\mu_1 Y-Z> 0\}$$ leaves $\mathcal{E}^+$ transversely through one of the following boundary faces
$$\{H=0\},\quad \{X_1-X_2=0\},\quad \{\mu_1Y-Z\}.$$
\end{proposition}
\begin{proof}
By \eqref{eqn_H'}, an integral curve in $(\mathcal{E}^+)^\circ$ must escape $\mathcal{E}^+$.
Since $\mathcal{RF}^+$, $\mathcal{E}\cap\{Y=0\}$, and $\mathcal{E}\cap\{Z=0\}$ are invariant, an integral curve in $(\mathcal{E}^+)^\circ$ can only exit $\mathcal{E}^+$ through one of the boundary faces listed above.

If an integral curve exits through $\{H=0\}$, transversality follows from \eqref{eqn_H'}.  
For $\{X_1-X_2=0\}$, differentiating gives
\begin{equation}
\label{eqn_X1-X2 deri}
\begin{split}
(X_1-X_2)'&=(X_1-X_2)H\left(G+\frac{1-H^2}{n}-1\right)+R_1-R_2\\
&=(X_1-X_2)H\left(G+\frac{1-H^2}{n}-1\right)+\frac{n+d_1}{d_2}A(Z-\mu_1 Y)(Z-\mu_2 Y),
\end{split}
\end{equation}
A non-transverse exit would force the integral curve to lie entirely in the invariant sets $\mathcal{E}\cap\{Y=Z=0\}$ or in $\Phi_i$, both disjoint from $(\mathcal{E}^+)^\circ$.  

Similarly, for $\{\mu_iY-Z=0\}$ one finds
\begin{equation}
\label{eqn_muiY-Z derivative}
(\mu_i Y-Z)'=(\mu_i Y-Z)H\left(G+\frac{1-H^2}{n}\right)-\mu_i Y X_1-Z (X_1-2X_2).
\end{equation}
So non-transverse escape is possible only in $\mathcal{E}\cap\{Y=Z=0\}$, or in $\Phi_i$. Again, these sets are disjoint from $(\mathcal{E}^+)^\circ$.  

Thus any integral curve in $(\mathcal{E}^+)^\circ$ escapes $\mathcal{E}^+$ only transversely.
\end{proof}

By \eqref{eqn_initial Z and H} and the coordinate of $p_0^+$, each $\gamma_s$ with $s>0$ is initially in $(\mathcal{E}^+)^\circ$. Therefore, each $\gamma_s$ with $s>0$ escapes $\mathcal{E}^+$ transversely by Proposition \ref{prop_each orbit transversely escapes}. 

Let $\sharp C(\gamma_s)$ be the number of times that a $\gamma_s$ intersects 
$$\mathcal{B}:=\mathcal{E}^+ \cap \{X_1-X_2=0\}\cap \{0<H<1\}.$$ 
We have the following proposition.
\begin{proposition}
\label{prop_bullet escape point}
If $\sharp C(\gamma_{s})\geq 1$ for some $s> 0$, the integral curve escapes $\mathcal{E}^+$ through the 2-dimensional region
$$
\tilde{\mathcal{B}}:=\mathcal{E}\cap \{X_1-X_2=0\}\cap \{0< H< 1\}\cap \{\mu_2 Y< Z < \mu_1 Y\}.
$$
\end{proposition}
\begin{proof}
We show that the escaping point of $\gamma_s$ necessarily contributes to $\sharp C(\gamma_{s})$. By Proposition \ref{prop_each orbit transversely escapes}, we need to rule out the other two possibilities.

If a $\gamma_{s}$ escapes through $\{H=0\}$, the function $H$ becomes and remains negative after the intersection. Thus $\sharp C(\gamma_s)=0$. 

If a $\gamma_s$ escapes through $\{\mu_1 Y-Z\}$, the integral curve enters the set 
$$
\mathcal{U}:=\mathcal{E}\cap \{\mu_1 Y-Z< 0\}\cap \{X_1-X_2>0\}\cap \{Z>0\}.
$$
By \eqref{eqn_muiY-Z derivative}, we have
$$\left.(\mu_1 Y-Z)'\right|_{\mu_1Y-Z=0}= 2Z(X_2-X_1)<0.
$$
By \eqref{eqn_X1-X2 deri} and $Z-\mu_2 Y>Z-\mu_1 Y>0$, we also have 
$$\left.(X_1-X_2)'\right|_{X_1-X_2=0}= \frac{n+d_1}{d_2}A(Z-\mu_1 Y)(Z-\mu_2 Y)>0.$$
The set $\mathcal{U}$ is invariant. So it is impossible for a $\gamma_s$ in $\mathcal{U}$ to intersect $\mathcal{B}$ afterward. 

The only possibility is $\gamma_{s}$ escaping $\mathcal{E}^+$ through $\{X_1-X_2=0\}$. As $X_1-X_2>0$ at $p_0^+$, we have 
$$\left.(X_1-X_2)'\right|_{X_1-X_2=0}=\frac{n+d_1}{d_2}A (Z-\mu_1 Y)(Z-\mu_2 Y)\leq 0$$ at the intersection point. The derivative cannot vanish at the intersection point, since $\gamma_{s}$ does not intersect $\Phi_i$. Therefore, we conclude that such a  $\gamma_{s}$ must intersect $\tilde{\mathcal{B}}$.
\end{proof}

We have the following lemma, which is an analogous statement to \cite[Theorem 4.6]{bohm_inhomogeneous_1998}.
\begin{lemma}
\label{lem_IVT}
If there is an $s_\bullet> 0$ such that $\sharp C(\gamma_{s_\bullet})\geq 1$, there exists some $s_\star\in (0,s_\bullet)$ such that $\gamma_{s_\star}$ is a heterocline that joins $p_0^\pm$. 
\end{lemma}
\begin{proof}
Since $\gamma_0$ is in the invariant set $\mathcal{E}\cap \{Z=0\}$, we have 
$$
(X_1-X_2)'=(X_1-X_2)H\left(G+\frac{1-H^2}{n}-1\right)+(d_1-1)Y^2
$$
on $\gamma_0$.
Hence, the function $X_1-X_2$, which is initially positive, can never vanish along $\gamma_0$. We have $C(\gamma_{0})=0$.

Let $\eta_s\in \mathbb{R}$ be the smallest number such that $\gamma_s(\eta_s)\in \partial \mathcal{E}^+$. The above argument shows that $\gamma_0(\eta_0)\in \mathcal{E}^+\cap \{H=0\}\cap \{X_1-X_2>0\}$. By Proposition~\ref{prop_bullet escape point}, we have $\gamma_{s_\bullet}(\eta_{s_\bullet})\in \tilde{\mathcal{B}}\subset\{X_1-X_2=0\}$. Since each $\gamma_s$ with $s>0$ escapes $\mathcal{E^+}$ transversely by Proposition \ref{prop_each orbit transversely escapes}, we have the following connected path
$$
\mathcal{P}:=\partial \mathcal{E}^+ \cap \{\gamma_s(\eta_s)\mid s\in[0,s_\bullet]\}
$$
that joins a point in $\mathcal{E}^+\cap \{H=0\}\cap \{X_1-X_2>0\}$ (from $\gamma_0$) and a point in $\tilde{\mathcal{B}}\subset\{X_1-X_2=0\}$ (from $\gamma_{s_\bullet}$). Since $\gamma_0(\eta_0)$ is not in the closure of $\tilde{\mathcal{B}}$. The path $\mathcal{P}$ intersects $\partial \tilde{\mathcal{B}}$. Therefore, there exists a $s_\star \in (0,s_\bullet)$ for which $\gamma_{s_\star}(\eta_{s_\star})\in \partial\tilde{\mathcal{B}}$. Since $\mathcal{RF}^+$ and each $\Phi_i$ are invariant, the point $\gamma_{s_\star}(\eta_{s_\star})$ must be in $\partial \tilde{B}\cap \{H=0\}\subset \{X_1=X_2=0\}$, meaning that $\gamma_{s_\star}$ escapes $\mathcal{E^+}$ through $\{X_1=X_2=0\}$. By the $\mathbb{Z}_2$-symmetry explained in Remark~\ref{rem_Z2}, such a $\gamma_{s_\star}$ is a heterocline that joins $p_0^\pm$. 
\end{proof} 

In \cite{chi2024positive}, we prove the existence of $s_\bullet$ by constructing a compact barrier set $\mathcal{S}$ that forces each $\gamma_s$ in the set to intersect $\mathcal{E}^+\cap\{H> 0\}\cap\{X_1-X_2=0\}$. With a different structural triple $(d_1,d_2,A)$, the set $\mathcal{S}$ may or may not exist. The counterpart to $\mathcal{S}$ in this paper depends on varying $(d_1,d_2,A)$. As $A$ falls below or exceeds certain thresholds determined by $(d_1, d_2)$, each $\gamma_s$ in the set is forced to escape through certain faces. 

Define functions
\begin{equation}
\label{eqn_barrier functions}
\begin{split}
S&:=\frac{Z}{Y}\frac{X_1+\frac{d_2}{2d_1}X_2}{X_2},\\
T&:=\frac{1-H^2}{n}\frac{1}{YZ},\\
P&:=X_1\left(R_2-\frac{1-H^2}{n}\right)-X_2\left(R_1-\frac{1-H^2}{n}\right)-2X_2(X_1-X_2)\left(X_1+\frac{d_2}{2d_1}X_2\right).
\end{split}
\end{equation}
The derivatives of functions \eqref{eqn_barrier functions} along $\gamma_s$ are
\begin{align}
\label{eqn_S'}
S'&=-\frac{1}{X_2\left(X_1+\frac{d_2}{2d_1}X_2\right)}SP\\
\label{eqn_T'}
T'&=2 TX_2\\
\label{eqn_P'}
P'&=P\left(H\left(3G+\frac{3}{n}(1-H^2)-1\right)+\frac{n}{d_1}X_2-X_1\right)+(X_1-X_2)(Q+P),
\end{align}
where 
\begin{equation}
\label{eqn_function Q}
\begin{split}
Q&=4X_2\left(X_1+\frac{d_2}{2d_1}X_2\right)\left(H+\frac{d_2}{2d_1}X_2\right)+\left(2X_1+2X_2+\frac{3d_2}{d_1}X_2\right)\frac{1-H^2}{n}\\
&\quad -2(d_2-1)X_1YZ-X_2\left(2(d_1-1)Y^2+\frac{3d_2}{d_1}(d_2-1)YZ\right).
\end{split}
\end{equation}

We discuss our motivation of using the functions listed in \eqref{eqn_barrier functions}. In the original coordinates, we have $T=\Lambda f_2^2$. From the analytic viewpoint, both $T$ and its derivative \eqref{eqn_T'} is positive in $(\mathcal{E}^+)^\circ$, making $T$ a convenient function to control along the integral curves. The function $S$ is also positive in $(\mathcal{E}^+)^\circ$, and its monotonicity depends on the sign of $P$. From the dynamical systems viewpoint, the function $S$ monitors the rotation of $\gamma_s$ around $\Phi_2$, which occurs in the $(X_1-X_2, \mu_2Y-Z)$–plane, while $H$ decreases monotonically. The ratio $\frac{Z X_1}{YX_2}$ is a natural quantity to record this rotation, but its derivative involves a factor (the original version of $P$) that does not preserve positivity in $(\mathcal{E}^+)^\circ$. The modified definition of $S$ is precisely tuned so that the derivative formulas \eqref{eqn_S'}–\eqref{eqn_P'} are better suited for our analysis.

From a geometric viewpoint, the function $S$ can be regarded as a perturbation of certain
first integrals that arise in the Ricci–flat case. For instance, many known cohomogeneity one Ricci–flat metrics with a first integral (e.g. the Bryant--Salamon $G_2$ and $\mathrm{Spin}(7)$ metrics \cite{bryant_construction_1989}) satisfy the following first order system in the original $(t,f_1,f_2)$-coordinates:
\begin{equation}
\label{eqn_first integral}
\frac{\dot{f_1}}{f_1}=\frac{1}{f_1}- \frac{d_2}{2d_1\sigma}\frac{f_1}{f_2^2},\quad 
\frac{\dot{f_2}}{f_2}=\frac{1}{\sigma}\frac{f_1}{f_2^2},\quad 
\sigma=\frac{d_1+1}{d_2-1}.
\end{equation}
In the new $(\eta,X_1,X_2,Y,Z)$-coordinates, this first integral is realized as an algebraic curve in $\mathcal{RF}^+$ given by
$$\mathcal{L}:=\mathcal{RF}^+\cap \{X_1 = Y- \tfrac{d_2}{2d_1\sigma}Z\}\cap \left\{X_2=\frac{Z}{\sigma}\right\}\cap \left\{\sigma =\frac{d_1+1}{d_2-1}\right\}.$$
We refer to Appendix~\ref{appendix:RFexamples} for the proof of $\mathcal{L}$ being invariant for some $(d_1,d_2,A)$ and explicit examples where the integrability condition can be verified. Our definition of $S$ unifies the two algebraic relations in $\mathcal{L}$ into the single condition $S=\sigma$. Perturbing $(d_1,d_2,A,\sigma)$ away from the Ricci–flat subsystem, together with the functions $P$ and $T$, provides a useful barrier in the positive Einstein setting. 

Define
\begin{equation}
\label{eqn_set_S}
\begin{split}
\mathfrak{S}_A(\sigma,\tau)&=\mathcal{E}^+\cap \{X_1\geq X_2\geq  0\} \cap \{P\geq 0\}\cap \{S\leq \sigma\}\cap \{T\geq \tau\},
\end{split}
\end{equation}
where we fix
$$
(\sigma,\tau)=\left(\frac{d_2(d_2-1)}{2d_1^2A},d_2-1-\frac{d_1+1}{d_2(d_2-1)}2d_1^2A\right).
$$
The set \eqref{eqn_set_S} is a closed subset of the compact set $\mathcal{E}^+$, see Proposition \ref{prop_S_is compact}, and it is therefore compact. 

\begin{remark}
\label{rem_sigma_tau}
In principle, other choices of $(\sigma,\tau)$ are possible to obtain an upper bound of existence for $A$. The present choice is artificially tuned so that $\gamma_{s_\bullet}$, where
\begin{equation}
\label{eqn_sbullet}
s_\bullet=\frac{d_1}{\tau}=\frac{d_1\sigma}{(d_2-1)\sigma-(d_1+1)},
\end{equation}
is the only candidate in the 1-parameter family to enter the set $\mathfrak{S}_A(\sigma,\tau)$
that plays the role in Lemma \ref{lem_IVT}. Furthermore, we show in \cite{chi2025nonexistence} that the present choice of $(\sigma,\tau)$ maximizes the admissible upper bound for $A$ under the framework of this paper.
\end{remark}

Here and in what follows, we take $s_\bullet$ as in \eqref{eqn_sbullet}. With this preparation, we first analyze the possible possible faces of $\mathfrak{S}_A(\sigma,\tau)$, through which an integral curve inside $\mathfrak{S}_A(\sigma,\tau)$ escapes. Next, we show that for suitable values of $A$, the integral curve $\gamma_{s_\bullet}$ indeed enters $\mathfrak{S}_A(\sigma,\tau)$ once it leaves $p_0^+$. Lastly, we identify ranges of $A$ for which escape through $\mathfrak{S}_A(\sigma,\tau)\cap \{P=0\}$ is impossible, forcing $\gamma_{s_\bullet}$ to leave through $\mathfrak{S}_A(\sigma,\tau)\cap\{X_1=X_2>0\}$ and hence producing the desired heterocline.

\begin{proposition}
\label{prop_how curves in triangle escape}
If $A\in \left[0,\frac{d_2(d_2-1)^2}{2d_1^2(d_1+1)}\right)$, a non-trivial integral curve in $\mathfrak{S}_A(\sigma,\tau)$ either escapes the set through the face $\mathfrak{S}_A(\sigma,\tau)\cap \{P=0\}$, or $\mathfrak{S}_A(\sigma,\tau)\cap\{X_1=X_2>0\}$.
\end{proposition}
\begin{proof}
Suppose a non-trivial integral curve escapes through the face $\{X_2=0\}$. If $X_1=0$ at the escaping point, there is nothing to prove since the escape point lies in $\{P=0\}$. If $X_1>0$ at the escaping point, the $Z$-coordinate must vanish as the function $S$ is non-increasing in $\mathfrak{S}_A(k,l)$. With $X_2=Z=0$, the inequality $P\geq 0$ becomes $-X_1\frac{1-d_1^2X_1^2}{n}\geq 0$, meaning that the escaping point can only be $p_0^+$, a contradiction.

By \eqref{eqn_basic inequalities for mu_i}, we have
\begin{equation}
\label{eqn_sigma is small}
\sigma=\frac{d_2(d_2-1)}{2d_1^2A}=\frac{n+d_1}{2d_1^2}\frac{\mu_1+\mu_2}{2}<\frac{n+d_1}{2d_1}\mu_1.
\end{equation}
By $X_1\geq X_2$ and $n=d_1+d_2$, we also have 
$$
\frac{Z}{Y}\frac{n+d_1}{2d_1}< \frac{Z}{Y}\frac{X_1+\frac{d_2}{2d_1}X_2}{X_2}=S.
$$
Therefore, the inequality $S\leq \sigma<\frac{n+d_1}{2d_1}\mu_1$ in $\mathfrak{S}_A(\sigma,\tau)$ implies 
$$\frac{Z}{Y}\frac{n+d_1}{2d_1}<\frac{n+d_1}{2d_1}\mu_1.$$ Hence, an integral curve in $\mathfrak{S}_A(\sigma,\tau)$ does not escape through the face $\{Z-\mu_1 Y=0\}$.

Since $S$ is non-increasing in $\mathfrak{S}_A(\sigma,\tau)$, an integral curve in $\mathfrak{S}_A(\sigma,\tau)$ does not escape the set through $\mathfrak{S}_A(\sigma,\tau)\cap \{S=\sigma\}\cap \{P>0\}$. The inequality $\tau>0$ holds if $A\in \left[0,\frac{d_2(d_2-1)^2}{2d_1^2(d_1+1)}\right).$ Since $T$ is non-decreasing in $\mathfrak{S}_A(\sigma,\tau)$, an integral curve in $\mathfrak{S}_A(\sigma,\tau)$ does not escape the set through $\mathfrak{S}_A(\sigma,\tau)\cap \{T=\tau\}$.
\end{proof}

Motivated by the derivative formula \eqref{eqn_S'}-\eqref{eqn_P'} and Proposition \ref{prop_how curves in triangle escape}, we next analyze the sign of $Q$ on $\mathfrak{S}_A(\sigma,\tau)\cap \{P=0\}$. This is crucial in ruling out escaping through $\{P=0\}$ for suitable ranges of $A$, and also helps clarify whether $\gamma_{s_\bullet}$ lies in $\mathfrak{S}_A(\sigma,\tau)$ once it leaves $p_0^+$. 

The set $\mathfrak{S}_A(\sigma,\tau)$ is equivalently defined as  
\begin{equation}
\begin{split}
\mathfrak{S}_A(\sigma,\tau)&=\left(\bigcup\limits_{(k,l)\in [0,1]\times [0,\mu_1]}\mathcal{E}_{k,l}\right)\cap \{P\geq 0\}\cap \{S\leq \sigma\}\cap \{T\geq \tau\},
\end{split}
\end{equation}
where
\begin{equation}
\label{eqn_kl slice}
\begin{split}
\mathcal{E}_{k,l}:=\mathcal{E}^+\cap \{X_2-kX_1=0,X_1\geq 0\}\cap \{Z-lY=0, Y\geq 0\}.
\end{split}
\end{equation}
Parametrize the algebraic surface $\mathfrak{S}_A(\sigma,\tau)\cap \{P=0\}$ using $(k,l)$ as in \eqref{eqn_kl slice}. The function $Q$ on $\mathfrak{S}_A(\sigma,\tau)\cap \{P=0\}$ is realized as a function on $\{(k,l)\mid (k,l)\in[0,1]\times [0,\mu_1]\}$. Specifically, with \eqref{eqn_conservation equation for 1}, the functions $P$ and $Q$ are equivalently defined as homogeneous polynomials of $(X_1,X_2,Y,Z)$ of degree $3$. On each slice $\mathcal{E}_{k,l}$, we have 
\begin{equation}
\label{eqn_P and Q in X and Y}
\begin{split}
P&=P_Y Y^2X_1+P_X X_1^3,\quad Q=Q_Y Y^2X_1+Q_X X_1^3,
\end{split}
\end{equation}
where
\begin{equation}
\begin{split}
P_Y(A,k,l)&=-\frac{d_1(n+d_1-2)+d_2(n+d_1-1)k}{d_2 (n-1)} Al^2\\
&\quad  + \frac{(d_2-1)(d_1+d_2k-1)}{n-1}l-\frac{(d_1-1)(d_1+d_2 k-k)}{n-1},\\
P_X(k)&=\frac{d_2(d_2-1)(d_1-1)k^2-d_1d_2 k^2 + 2(d_2 - 1)(d_1 -1)d_1k + d_1^2(d_1-1)}{d_1(n - 1)}(1 - k),\\
Q_Y(A,k,l)&=-\frac{2 n k+d_2 k+2 d_1}{n-1}Al^2\\
&\quad -\frac{2d_1(d_1-1)+d_1d_2 k-3d_2 k}{d_1(n-1)}(d_2-1)l+\frac{(2d_1+d_2 k+2 k)(d_1-1)}{n-1},\\
Q_X(k)&=4k\left(1+\frac{d_2 k}{2d_1}\right)\left(d_1+d_2k+\frac{d_2k}{2d_1}\right)+\left(2+2k+\frac{3d_2k}{d_1}\right)\frac{d_1+d_2k^2-(d_1+d_2k)^2}{n-1}.
\end{split}
\end{equation}

\begin{proposition}
\label{prop_px is negative somewhere}
For any $d_1, d_2\geq 2$, the inequality $P_X\geq 0$ holds on $[0,1]$ and only vanishes at $k=1$.
\end{proposition}
\begin{proof}
The quadratic factor of $P_X$ satisfies the inequality
\begin{equation}
\begin{split}
&d_2(d_2-1)(d_1-1)k^2-d_1d_2 k^2 + 2(d_2 - 1)(d_1 -1)d_1k + d_1^2(d_1-1)\\
&\geq d_2(d_2-1)(d_1-1)k^2 + (2(d_2 - 1)(d_1 -1)d_1-d_1d_2)k + d_1^2(d_1-1).
\end{split} 
\end{equation}
The coefficient of each $k^i$ is non-neagative on $[0,1]$ since $d_1,d_2\geq 2$. Therefore, the polynomial is positive on $[0,1]$.
\end{proof}

Define $\omega=Q_YP_X-Q_XP_Y$. On $\mathfrak{S}_A(\sigma,\tau)\cap \{P=0\}$, we have 
\begin{equation}
\label{eqn_Q and omegaA}
\begin{split}
Q&= Q_Y Y^2X_1+Q_X \frac{P-P_YY^2X_1}{P_X}=\frac{\omega}{P_X}Y^2X_1.
\end{split}
\end{equation}
The function $\omega$ has the same sign as $Q$ on $\mathfrak{S}_A(\sigma,\tau)\cap \{P=0\}$. The specific formula for $\omega$ is 
\begin{equation}
\label{eqn_omega formula}
\begin{split}
\omega(A,k,l)&=\frac{1}{d_1^2(n-1)}\left(\frac{2d_1+d_2k}{d_2}\omega_2 Al^2+(d_2-1)k\omega_1 l -(d_1-1)k^2\omega_0\right),
\end{split}
\end{equation}
where
\begin{equation}
\begin{split}
\omega_2(k)&=(2d_1^2d_2^2 - d_1d_2^3 + d_2^3 - d_2^2)k^3 + (4d_1^3d_2 - 4d_1^2d_2^2 - 2d_1^2d_2 + 4d_1d_2^2 - 2d_1d_2)k^2\\
&\quad + (2d_1^4 - 5d_1^3d_2 - 2d_1^3 + 5d_1^2d_2)k - 2d_1^4 + 2d_1^3,\\
\omega_1(k)&=(d_1d_2^3 - 4d_1d_2^2 - d_2^3 + 3d_2^2)k^3 + (2d_1^2d_2^2 - 8d_1^2d_2 + 8d_1d_2 - 2d_2^2)k^2 \\
&\quad + (d_1^3d_2 - 4d_1^3 + 5d_1^2d_2 + 4d_1^2 - 6d_1d_2)k + 4d_1^3 - 4d_1^2,\\
\omega_0(k)&=(d_1d_2^3 - 2d_1d_2^2 - d_2^3 - 2d_1d_2 + d_2^2)k^2 \\
&\quad + (2d_1^2d_2^2 - 2d_1^2d_2 - 2d_1d_2^2 - 4d_1^2 + 4d_1d_2)k + d_1^3d_2 - d_1^2d_2 + 4d_1^2.
\end{split}
\end{equation}

The critical point $p_0^+$, whose $(k,l)$-coordinate is $(0,0)$, is on the boundary of $\mathfrak{S}_A(\sigma,\tau)\cap \{P=0\}$. If $A=0$, the product Einstein metric is represented by the algebraic curve $\mathcal{E}\cap \{X_2=0\}\cap \left\{R_2-\frac{1-H^2}{n}=0\right\}$, at which both $P$ and $Q$ identically vanish. This is reflected by $\omega(0,0,l)$ being identically zero. 

For any $A\geq 0$, both $\omega(A,k,l)$ and its gradient vanish at $(0,0)$. To determine the sign of $\omega(A,k,l)$ near $p_0^+$, we compute the Hessian of the function at $(0,0)$,
\begin{equation}
\label{eqn_Hessian}
\left.\mathrm{Hess}(\omega)\right|_{(0,0)}=
\begin{bmatrix}
-\frac{2(d_1-1)(d_1d_2-d_2+4)}{n-1}&\frac{4(d_1-1)(d_2-1)}{n-1}\\
\frac{4(d_1-1)(d_2-1)}{n-1}&-\frac{8d_1^2(d_1-1)}{d_2(n-1)}A
\end{bmatrix}.
\end{equation}
By the second derivative test, we have the following proposition.
\begin{proposition}
\label{prop_local sign of omega}
If $A\in \left[0,\frac{d_2(d_2-1)^2}{d_1^2(d_1d_2-d_2+4)}\right)$, the critical point $(0,0)$ is a saddle. If $A>\frac{d_2(d_2-1)^2}{d_1^2(d_1d_2-d_2+4)}$, the critical point $(0,0)$ is a local maximum.
\end{proposition}

Note that $\frac{d_2(d_2-1)^2}{d_1^2(d_1d_2-d_2+4)}$ is not larger than the upper bound in Proposition \ref{prop_how curves in triangle escape} provided $d_1,d_2 \geq 2$. In particular, Propositions \ref{prop_how curves in triangle escape} and \ref{prop_local sign of omega}, together with the derivative formula \eqref{eqn_P'}, show that a necessary condition for $Q\geq 0$ on $\mathfrak{S}_A(\sigma,\tau)\cap\{P=0\}$ is 
$A\in \left[0,\tfrac{d_2(d_2-1)^2}{d_1^2(d_1d_2-d_2+4)}\right).$

From now on we restrict to this range of $A$. To prove that $\gamma_{s_\bullet}$ indeed enters $\mathfrak{S}_A(\sigma,\tau)$, we examine in turn the defining inequalities of $\mathfrak{S}_A(\sigma,\tau)$. 
The bounds involving $S\leq \sigma$ and $T\geq \tau$ are immediate from the linearized solution near $p_0^+$ and the monotonicity of these functions, as shown in the following proposition.

\begin{proposition}
\label{prop_ST_initial}
Along the integral curve $\gamma_{s_\bullet}$ one has $T\geq\tau$ once the curve leaves $p_0^+$. One also has $S\leq\sigma$ initially provided that $P>0$ initially.
\end{proposition}
\begin{proof}
By \eqref{eqn_linearized solution near p0+}, we have 
$$
\lim\limits_{\eta\to-\infty} S(\gamma_{s_\bullet})=\sigma,\quad \lim\limits_{\eta\to-\infty} T(\gamma_{s_\bullet})=\tau.
$$
By definitions, we have $S,T> 0$ in $(\mathcal{E}^+)^\circ$.
Moreover, by \eqref{eqn_T'},  the function $T$ is monotonic in $(\mathcal{E}^+)^\circ$. By \eqref{eqn_S'}, the function $S$ is monotonic provided that $P>0$. Therefore, the inequalities $S\leq \sigma$ and $T\geq \tau$ hold along $\gamma_{s_\bullet}$ once the integral curve leaves $p_0^+$, given that $P>0$ initially.
\end{proof}

\begin{remark}
\label{rem_motivation of S and T}
The rational functions $S$ and $T$ are not defined at $p_0^+$ since $X_2=Z=0$ at $p_0^+$. Nevertheless, the conditions $S=\sigma$ and $T=\tau$ extend smoothly across $p_0^+$ via the relation 
$$\tilde{S}:=Z\left(X_1+\frac{d_2}{2d_1}X_2\right)-\sigma YX_2=0,\quad \tilde{T}:=\frac{1-H^2}{n}-\tau YZ=0.$$
Consider the unstable eigenvector $s_\bullet v_1+v_2$, along which $\gamma_{s_\bullet}$ emanates $p_0^+$. Since
$$\nabla\tilde{S}(p_0^+)=\begin{bmatrix}
0\\
-\frac{\sigma}{d_1}\\
0\\
\frac{1}{d_1}
\end{bmatrix},\quad \nabla \tilde{T}(p_0^+)=\begin{bmatrix}
-\frac{2d_1}{n}\\
-\frac{2d_2}{n}\\
0\\
-\frac{\tau}{d_1}
\end{bmatrix},$$
we have 
$$
\langle \nabla\tilde{S}(p_0^+),s_\bullet v_1+v_2 \rangle=\langle \nabla\tilde{T}(p_0^+),s_\bullet v_1+v_2\rangle=0.
$$
Under this smooth extension, the unstable eigenvector is tangent to the hypersurfaces $\{S=\sigma\}$ and $\{T=\tau \}$ at $p_0^+$. We work with $S$ and $T$ (instead of $\tilde{S}$ and $\tilde{T}$) because their derivatives along integral curves, \eqref{eqn_S'}–\eqref{eqn_T'}, are simpler and allow one to determine the initial sign of $S-\sigma$ and $T-\tau$ along $\gamma_{s_\bullet}$. In contrast, the linearized solution provides no information about initial signs of $\tilde{S}$ and $\tilde{T}$, so they are less effective for our purpose.
\end{remark}

It is left to show that $P>0$ initially along $\gamma_{s_\bullet}$. Since 
$$\nabla P(p_0^+)=\begin{bmatrix}
\frac{2}{n}\\
\frac{d_1(d_2-d_1)-n}{nd_1^2}\\
0\\
\frac{d_2-1}{d_1^2}
\end{bmatrix}$$
is perpendicular to both $v_1$ and $v_2$, the initial sign of $P$ along $\gamma_s$ cannot be read off from the linearized solution \eqref{eqn_linearized solution near p0+} (similar to the situation for $\tilde{S}$ and $\tilde{T}$ in Remark \ref{rem_motivation of S and T}). To resolve this, we consider the auxiliary function $\omega$, whose sign controls the initial behavior of $P$, as shown in the next proposition.

\begin{proposition}
\label{prop_P and omega have the same sign initially}
If $\omega$ is not identically zero along a $\gamma_s$ with $s> 0$, the function $\omega$ has the same sign as $P$ once the integral curve leaves $p_0^+$.
\end{proposition}
\begin{proof}
Recall that the $(k,l)$-coordinate of $p_0^+$ is $(0,0)$, and we have
$$
P_X(0)=\frac{d_1(d_1-1)}{n-1}>0.
$$
Hence for any $\gamma_s$ with $s> 0$, the integral curve initially satisfies $P_X>0$ in a small neighborhood of $p_0^+$.
We have 
\begin{equation}
\label{eqn_Q+P}
\begin{split}
Q+P
%&= Q_Y Y^2X_1+Q_X \frac{P-P_YY^2X_1}{P_X}+P
=\frac{\omega}{P_X}Y^2X_1+\frac{P_X+Q_X}{P_X}P.
\end{split}
\end{equation}
By Proposition \ref{prop_PX+QX is negative} we know that $P_X+Q_X<0$ for each $k\in [0,1]$. 

Consider first the case $\omega>0$ once $\gamma_s$ leaves $p_0^+$. If we assume $P\leq 0$ initially, then the two terms on the right-hand side of \eqref{eqn_Q+P} are both non-negative: the first term is strictly positive (since $\omega>0$, $P_X>0$ and $Y^2X_1> 0$ near $p_0^+$), while the second term is non-negative because $\frac{P_X+Q_X}{P_X}<0$ and $P\leq 0$. Thus $Q+P>0$ initially.

Next we examine the sign of $P'$ near $p_0^+$. We claim that the first term in \eqref{eqn_P'} is non-negative, since we assume $P\leq 0$ and
\begin{equation}
\label{eqn_a factor for P'}
\begin{split}
&\left(H\left(3G+\frac{3}{n}(1-H^2)-1\right)+\frac{n}{d_1}X_2-X_1\right)(\gamma_s)\\
&\sim \frac{2}{d_1}-1- ((12+6d_2)+4(d_2-1)(2d_2-1)s)e^{\frac{2}{d_1}\eta}+O\left(e^{\left(\frac{2}{d_1}+\delta\right)\eta}\right)\\
&<0
\end{split}
\end{equation}
near $p_0^+$. Together with $Q+P>0$ and $X_1-X_2=\frac{1}{d_1}>0$ initially (which makes the second term in \eqref{eqn_P'} positive), this implies that $P'$ is initially strictly positive, contradicting the assumption $P\leq 0$ initially.

The argument for the case $\omega<0$ is analogous: if $P\geq 0$ initially then by \eqref{eqn_Q+P} we obtain $Q+P<0$ initially. The inequality \eqref{eqn_a factor for P'} still holds. By the assumption that $P\geq 0$, the first term in \eqref{eqn_P'} is non-positive. Since $X_1-X_2=\frac{1}{d_1}$ and $Q+P<0$ initially, the second term in \eqref{eqn_P'} is negative. Hence, the derivative $P'$ is negative initially, a contradiction.

Therefore, whenever $\omega$ is not identically zero along $\gamma_s$, the sign of $\omega$ agrees with the sign of $P$ once the integral curve leaves $p_0^+$.
\end{proof}

\begin{proposition}
\label{prop_initial_barrier}
The integral curve $\gamma_{s_\bullet}$ is initially in $\mathfrak{S}_A(k,l)$ if $A\in \left (0,\frac{d_2(d_2-1)^2}{d_1^2(d_1d_2-d_2+4)}\right)$.
\end{proposition}
\begin{proof}
In terms of $(k,l)$, the relation $S=\sigma$ is equivalent to $$l=\sigma\frac{2d_1k}{2d_1+d_2k}.$$ By Remark \ref{rem_motivation of S and T}, the integral curve $\gamma_{s_\bullet}$ is tangent to $\{S=\sigma\}$ after smooth extension at $p_0^+$. Therefore, the $(k,l)$-projection of $\gamma_{s_\bullet}$ is tangent to $v=(1,\sigma)$ at $(0,0)$. 
Recall that $\omega$ and its gradient vanish at $(0,0)$. By \eqref{eqn_Hessian}, we have 
$$\left.\mathrm{Hess}(\omega)\right|_{(0,0)}(v,v)=\frac{2(d_1-1)}{n-1}\left(\frac{d_2(d_2-1)^2}{d_1^2}\frac{1}{A}-(d_1d_2-d_2+4)\right)>0$$
given that $A<\frac{d_2(d_2-1)^2}{d_1^2(d_1d_2-d_2+4)}$. Hence, the function $\omega$ is positive initially along $v$. By Proposition \ref{prop_P and omega have the same sign initially}, the inequality $P>0$ holds initially along $\gamma_{s_\bullet}$. Together with Proposition \ref{prop_ST_initial}, the integral curve $\gamma_{s_\bullet}$ is initially in $\mathfrak{S}_A(\sigma,\tau)$.
\end{proof}

Proposition \ref{prop_initial_barrier} provides a condition on $A$ to form a barrier for $\gamma_{s_\bullet}$ locally, which means that $\gamma_{s_\bullet}$ may still escape $\mathfrak{S}_A(\sigma,\tau)$ through $\mathfrak{S}_A(\sigma,\tau)\cap \{P=0\}$ eventually. In the following, we explore the condition on $A$ to preserve the positivity of $P$ (or equivalently the positivity of $\omega$) when $\gamma_{s_\bullet}$ is in $\mathfrak{S}_A(\sigma,\tau)$, making $\mathfrak{S}_A(\sigma,\tau)$ a global barrier so that $\gamma_{s_\bullet}$ can only escape through $\mathfrak{S}_A(\sigma,\tau)\cap \{X_1=X_2>0\}$.

Consider $\omega(A,k,l)$ as a quadratic function of $l$. We have the following proposition for the coefficients of $\omega$.
\begin{proposition}
\label{prop_omega_i have sign for positive k}
For any $d_1, d_2 \geq 2$, the inequalities $\omega_2<0$ and $\omega_0> 0$ hold for any $k\in[0,1]$.
\end{proposition}
\begin{proof}
Since 
$$\omega_2\left(-\frac{d_1}{d_2}\right)=-\frac{d_1^3}{d_2}<0,\quad \frac{d \omega_2}{dk}\left(-\frac{d_1}{d_2}\right)=d_1^2(2d_1+1)>0,$$
$$\omega_2(0)=-2d_1^3(d_1-1)<\omega_2\left(-\frac{d_1}{d_2}\right),$$ 
the function $\omega_2$ changes monotonicity at least once on $\left[-\frac{d_1}{d_2},0\right]$. Since $\omega_2$ changes monotonicity at most twice, the function decreases and then increases, or monotonically increases on $[0,1]$. From $\omega_2(1)=-d_2(n-1)(d_1n-n-d_1)< 0$ we know that $\omega_2<0$ on $[0,1]$.

For $\omega_0$, we have 
$$
\frac{d\omega_0}{dk}(0)=2d_1(d_2-2)(d_1d_2-d_2+d_1)\geq 0.
$$
The function either monotonically increases, or increases and then decreases on $[0,1]$. Since
$$\omega_0 (0)=(d_1-1)d_1^2d_2+4d_1^2>0,\quad \omega_0(1)=d_2(n-1)(d_1n-n-d_1)>0,$$
the positivity of $\omega_0$ on $[0,1]$ is established.
\end{proof}

By Proposition \ref{prop_omega_i have sign for positive k}, we learn that for each fixed $A>0$ and $k\in (0,1)$, the quadratic function $\omega(A,k,l)$ is concave down. By \eqref{eqn_Q and omegaA} and Proposition \ref{prop_px is negative somewhere}, the positivity of $Q$ on $\mathfrak{S}_A(\sigma,\tau)\cap \{P=0\}\cap\{X_1>X_2>0\}$ is implied by the following condition: for each $k\in (0,1)$, the quadratic $\omega(A,k,l)$ has two real roots $\mu_2(A,k)<\mu_1(A,k)$ and the interval for $l$ obtained from $\mathfrak{S}_A(\sigma,\tau)\cap \{P=0\}$ is contained in $(\mu_2(A,k),\mu_1(A,k))$.

In the following, we specify how $\mathfrak{S}_A(\sigma,\tau)\cap \{P=0\}$ gives rise to a closed interval for $l$. In terms of $(k,l)$, the inequality $S\leq \sigma$ is equivalent to 
\begin{equation}
l\leq \nu_1(A,k):= \sigma \frac{2d_1 k}{2d_1+d_2k}=\frac{d_2(d_2-1)}{d_1A}\frac{k}{2d_1+d_2k},
\end{equation}
and it gives the upper bound for $l$. Presumably the lower bound $\nu_2(A,k)$ for $l$ is obtained from $\{P=0\}\cap \{T\geq \tau\}$. It takes the following computations to write $T\geq \tau$ in terms of $k$ and $l$. From \eqref{eqn_conservation equation for 1}, the inequality $T\geq \tau$ is equivalent to 
$$G-H^2+d_1R_1+d_2R_2-(n-1)\tau YZ\geq 0.$$
Rewrite the above inequality as $T_X(k)X_1^2+T_Y(A,k,l)Y^2\geq 0,$ where $T_X$ and $T_Y$ are polynomials.
With $P=0$, the condition $T\geq \tau$ is equivalent to
\begin{equation}
\begin{split}
T_XX_1^2+T_YY^2=\frac{T_Y P_X-T_X P_Y}{P_X}Y^2\geq 0.
\end{split}
\end{equation}
Define $\zeta(A,k,l)=T_Y P_X-T_X P_Y$. The specific formula for $\zeta(A,k,l)$ is 
\begin{equation}
\zeta(A,k,l)=\zeta_2Al^2+\zeta_1 l+ \zeta_0,
\end{equation}
where
\begin{equation}
\begin{split}
\zeta_2(k)&=-\frac{(d_1+d_2 k)(d_1+d_2 k-1)(2d_1+d_2k)}{d_2}\\
\zeta_1(k,\tau)&=\frac{d_2-1}{d_1}(d_1+d_2k)(d_2k^2+d_2(d_1-1)k+d_1(d_1-1))\\
&\quad -\frac{1-k}{d_1}((d_1 d_2^2 - 2 d_1 d_2 - d_2^2 + d_2)k^2 + (2 d_1^2 d_2 - 2 d_1^2 - 2 d_1 d_2 + 2 d_1)k + d_1^2(d_1-1))\tau\\
\zeta_0(k)&=-k(d_1-1)(d_1+d_2k)(d_1+d_2k+1-2k).
\end{split}
\end{equation}

As $P_X>0$ for each $k\in(0,1)$, the function $T-\tau$ and  $\zeta$ have the same sign.
We show in the following proposition that for any $A\in\left(0,\frac{d_2(d_2-1)^2}{d_1^2(d_1d_2-d_2+4)}\right)$, the condition $\zeta(A,k,l)\geq 0$ always yields the lower bound for $l$ with each fixed $k\in (0,1)$.
\begin{proposition}
\label{prop_nu2<nu_1}
For each $(A,k)\in\left(0,\frac{d_2(d_2-1)^2}{d_1^2(d_1d_2-d_2+4)}\right)\times (0,1)$, the quadratic function $\zeta$ has two positive real roots. The smaller real root is in $(0,\nu_1(A,k))$.
\end{proposition}
\begin{proof}
A straightforward computation shows
\begin{equation}
\begin{split}
\zeta\left(A,k,\nu_1(A,k)\right)
%=\zeta\left(A,k,\frac{d_2(d_2-1)}{d_1 A} \frac{k}{2d_1+d_2k}\right)
=\frac{k^2}{d_1^2(2d_1+d_2k)A}(d_2(d_2-1)^2\beta_0 - d_1^2 \beta_1 A),
\end{split}
\end{equation}
 where
\begin{equation}
\label{eqn_beta0 and beta1}
\begin{split}
\beta_0(k)&=(d_1 d_2^2 - 2 d_1 d_2 - d_2^2 + d_2) k^2 + (2 d_1^2 d_2 - 2 d_1^2 - d_1 d_2 + 2 d_1) k + d_1^3 - d_1,\\
\beta_1(k)&=(2 d_1^2 d_2^2 + d_1 d_2^3 - 4 d_1^2 d_2 - 2 d_1 d_2^2 -  d_2^3 - 2 d_1 d_2 + 2  d_2) k^2 \\
&\quad + (4 d_1^3 d_2 + 2 d_1^2 d_2^2 - 4 d_1^3 - 2 d_1^2 d_2 - 3 d_1 d_2^2 + 4 d_1 d_2 +  d_2^2 + 4 d_1 - 2 d_2) k\\
&\quad  + 2 d_1^4 + d_1^3 d_2 - 2 d_1^2 d_2 + 2 d_1^2 + d_1 d_2 - 4 d_1.
\end{split}
\end{equation}
The coefficients $\beta_0$ and $\beta_1$ are positive for each $k\in [0,1]$ by Proposition \ref{prop_beta0 and beta_1}. Furthermore, from Proposition \ref{prop_min of beta0 over beta1} we know that 
$$\min\limits_{k\in[0,1]}\left\{\frac{d_2(d_2-1)^2}{d_1^2}\frac{\beta_0}{\beta_1}\right\}=\frac{d_2(d_2-1)^2}{d_1^2(n+d_1)}\geq \frac{d_2(d_2-1)^2}{d_1^2(d_1d_2-d_2+4)}.$$ 
Therefore, the inequality $\zeta\left(A,k,\sigma\frac{2d_1k}{2d_1+d_2k}\right)>0$ is valid for any $(A,k)\in\left(0,\frac{d_2(d_2-1)^2}{d_1^2(d_1d_2-d_2+4)}\right)\times (0,1)$. As $\zeta_2,\zeta_0<0$ for any $k\in(0,1)$, we know that $\zeta(A,k,l)$ has two real roots for each $k\in(0,1)$. Furthermore, the smaller root is less than $\nu_1(A,k)$.
\end{proof}

Denote the smaller positive root of $\zeta(A,k,l)$ as $\nu_2(A,k)$. Proposition \ref{prop_nu2<nu_1} guarantees the existence of the closed interval $[\nu_2(A,k),\nu_1(A,k)]$ for any $(A,k)\in\left(0,\frac{d_2(d_2-1)^2}{d_1^2(d_1d_2-d_2+4)}\right)\times (0,1)$. 

\begin{proposition}
\label{prop_mu2<nu1<mu1}
Define $A_1=\min\limits_{k\in [0,1]} F$,
where 
$
F:=\frac{d_2(d_2-1)^2}{d_1^2(2d_1+d_2k)}\frac{\beta_2}{\omega_0}
$, and
$$
\beta_2=(d_1d_2^3 - 2d_1d_2^2 - d_2^3 + d_2^2)k^3 + (2d_1^2d_2^2 - 4d_1^2d_2 - 2d_1d_2^2 + 2d_1d_2)k^2 + (d_1^3d_2 - 2d_1^3 + d_1^2d_2)k + 2d_1^3.
$$
The number $A_1$ is positive and the function $\omega\left(A,k,\nu_1(A,k)\right)$ is positive for any $k\in (0,1)$ if  $A\in(0,A_1)$.
\end{proposition}
\begin{proof}
With the help of Maple, we compute
\begin{equation}
\label{eqn_evaluation of omega at nu1}
\begin{split}
&\omega\left(A,k,\nu_1(A,k)\right)
%=\omega\left(A,k,\frac{d_2(d_2-1)}{d_1 A} \frac{k}{2d_1+d_2k}\right)
=\frac{(d_1-1)k^2}{d_1^2(n-1)}\left(\frac{d_2(d_2-1)^2}{d_1^2(2d_1+d_2k)}\frac{\beta_2}{A}-\omega_0\right)\\
%&=\frac{1}{d_1^2(n-1)}\left(\frac{2d_1+d_2k}{d_2}\omega_2 A\left(\frac{d_2(d_2-1)}{d_1 A} \frac{k}{2d_1+d_2k}\right)^2+(d_2-1)k\omega_1 \left(\frac{d_2(d_2-1)}{d_1 A} \frac{k}{2d_1+d_2k}\right) -(d_1-1)k^2\omega_0\right)
\end{split}
\end{equation}

By Proposition \ref{prop_beta1} in the Appendix, the factor $\frac{d_2(d_2-1)^2}{d_1^2(2d_1+d_2k)}\beta_2$ is positive for any $k\in [0,1]$. As $\omega_0$ is also positive by Proposition \ref{prop_omega_i have sign for positive k}, we know that $A_1>0$. The proof is complete.
\end{proof}

%\begin{remark}
%\label{rem_A1}
%From Proposition \ref{prop_formula for A1}, we know that the function $\frac{d_2(d_2-1)^2}{d_1^2(2d_1+d_2k)}\frac{\beta_2}{\omega_0}$ increases on $[0,1]$ if $d_1\geq 4$ and we have $A_1=\frac{d_2(d_2-1)^2}{d_1^2(d_1d_2-d_2+4)}$ for these cases. If $(d_1,d_2)=(2,2)$, the function increases and then decrease on $[0,1]$ and we still have $A_1=\frac{d_2(d_2-1)^2}{d_1^2(d_1d_2-d_2+4)}=\frac{1}{12}$. With $d_1\in\{2,3\}$ and $d_2\geq 3$, the function decreases and then increases. We have $A_1(d_1,d_2)<\frac{d_2(d_2-1)^2}{d_1^2(d_1d_2-d_2+4)}$ for these cases. We leave the complicated formulas for $A_1(2,d_2)$ and $A_1(3,d_2)$ in \eqref{eqn_A1 d1=2,3} in the Appendix.
%\end{remark}

Since $\omega(A,k,l)$ is concave down and $\omega(A,k,0)<0$ for any $k\in (0,1)$, Proposition \ref{prop_mu2<nu1<mu1} implies that an $\omega(A,k,l)$ with with $(A,k)\in(0,A_1)\times (0,1)$ has two positive roots $\mu_2(A,k)<\mu_1(A,k)$ and $\nu_1(A,k)\in (\mu_2(A,k), \mu_1(A,k))$ for any $k\in (0,1)$. 

With our choice of $(\sigma,\tau)$, it is clear that 
$$\nu_2(A,0)=\nu_1(A,0)=0,\quad \frac{d\nu_2}{d k}(A,0)=\frac{d\nu_1}{dk}(A,0)=\frac{d_2(d_2-1)}{2d_1^2A}.$$
Therefore, the inclusion $[\nu_2(A,k),\nu_1(A,k)]\subset (\mu_2(A,k),\mu_1(A,k))$ holds with any sufficiently small $k>0$. Therefore, the inclusion can be extended to all $k\in(0,1)$ if $\nu_2(A,k)>\mu_2(A,k)$ holds for all $k\in (0,1)$. Such a condition puts another restriction on the upper bound for $A$.

\begin{proposition}
\label{prop_nu_2>mu_2}
There exists an $A_2>0$ such that for any $A\in(0,A_2)$, the inequality $\nu_2(A,k)>\mu_2(A,k)$ holds for any $k\in (0,1)$.
\end{proposition}
\begin{proof}
We compute the resultant of $(\omega,\zeta)$.
\begin{equation}
\begin{split}
\mathrm{Res}_l(\omega,\zeta)&=-\frac{(d_1-1)(2d_1+d_2k)(1-k)^2k^2\beta_3^2}{d_1^6d_2^3(d_2-1)^2(n-1)^2}A \Theta_{d_1,d_2},
\end{split}
\end{equation}
where
\begin{equation}
\begin{split}
\beta_3&=d_2(d_1-1)(d_2-2)k^2-d_2k^2+2d_1(d_1-1)(d_2-1)k+d_1^2(d_1-1)\\
&>-d_2k+2d_1(d_1-1)(d_2-1)k\\
&>0,
\end{split}
\end{equation}
and 
\begin{equation}
\begin{split}
\Theta_{d_1,d_2}(A,k)&=\theta_2 A^2+\theta_1 A+\theta_0,
\end{split}
\end{equation}
with coefficients $\theta_i(k)$ be polynomials; see Proposition \ref{prop_formula_of_Theta}. Furthermore, by Proposition \ref{prop_coefficients for zeta}, we have $\theta_2<0$ and $\theta_0\geq 0$ for any $k\in [0,1]$. The coefficient $\theta_0$ only vanishes at $k=0$. 
We also have 
\begin{equation}
\begin{split}
&\Theta_{d_1,d_2}(A,0)\\
&=-(d_1^3d_2-d_1^2d_2+4d_1^2)(4d_1^4-2d_1^3) A^2+4d_1^2d_2(d_2-1)^2 (d_1+1)(2d_1^7-2d_1^5)A.
\end{split}
\end{equation}
Therefore, for any fixed $k\in [0,1]$, the quadratic polynomial $\Theta_{d_1,d_2}(A,k)$ always has exactly one positive root $A^+(k)$. 
Define $A_2=\min\limits_{k\in[0,1]} A^+(k)>0$. If $A\in\left(0,A_2\right)$, the factor $\Theta_{d_1,d_2}(A,k)$ is positive for any $k\in [0,1]$. Therefore, the resultant $\mathrm{Res}_l(\omega,\zeta)$ does not vanish for any $(A,k)\in(0,A_2)\times (0,1)$. As $\nu_2(A,k)>\mu_2(A,k)$ for a sufficiently small $k>0$, we conclude that $\nu_2(A,k)>\mu_2(A,k)$ for any $k\in (0,1)$.
\end{proof}

Define $\chi_{d_1,d_2}=\min\left\{\frac{d_2(d_2-1)^2}{d_1^2(d_1d_2-d_2+4)},A_1,A_2\right\}$. We obtain the following lemma.
\begin{lemma}
\label{lem_the set is good}
Consider a structural triple $(d_1,d_2,A)$ with $A<\chi_{d_1,d_2}$. The integral curve $\gamma_{s_\bullet}$ joins $p_0^\pm$ or $\sharp C(\gamma_{s_\bullet})\geq 1$.
\end{lemma}
\begin{proof}
Since $\chi_{d_1,d_2}\leq \frac{d_2(d_2-1)^2}{d_1^2(d_1d_2-d_2+4)}$, the integral curve $\gamma_{s_\bullet}$ is initially in $\mathfrak{S}_A(\sigma,\tau)$ by Proposition \ref{prop_initial_barrier}.
Since $\chi_{d_1,d_2}\leq A_1,A_2$, the interval $[\nu_2(A,k),\nu_1(A,k)]$ exists and $[\nu_2(A,k),\nu_1(A,k)]\subset (\mu_2(A,k),\mu_1(A,k))$ for any $k\in (0,1)$ by Proposition \ref{prop_nu2<nu_1}, Proposition \ref{prop_mu2<nu1<mu1}, and Proposition \ref{prop_nu_2>mu_2}. Therefore, the inequality $Q>0$ holds on $\mathfrak{S}_A(\sigma,\tau)\cap \{P=0\}\cap\{X_1>X_2>0\}$. By \eqref{eqn_S'}-\eqref{eqn_T'} and Proposition \ref{prop_how curves in triangle escape}, the integral curve $\gamma_{s_\bullet}$ intersects $\mathfrak{S}_A(\sigma,\tau)\cap \{X_1=X_2\geq 0\}$.
\end{proof}

Lemma \ref{lem_IVT} and Lemma \ref{lem_the set is good} prove Theorem \ref{thm_main 1}.

\appendix
\section{Ricci–flat metrics with first integrals}
\label{appendix:RFexamples}
\begin{theorem}
\label{thm_integrable RF}
If $(d_1,d_2,A)$ satisfies 
\begin{equation}
\label{eqn_suffices to get a 1st integral}
A=\frac{d_2(d_2-1)^2}{4d_1^2(d_1+1)^2}(d_2+4d_1-d_1d_2),
\end{equation}
the set $\mathcal{L}$ is invariant. The Ricci-flat metric represented by $\gamma_\infty$ is integrable.
\end{theorem}
\begin{proof}
With $H=1$ in $\mathcal{RF}^+$, we have
\begin{equation}
\label{eqn_RF deri1}
\begin{split}
&\left.\left(X_1-Y+\frac{d_2}{2d_1}\frac{d_2-1}{d_1+1}Z\right)'\right|_{\mathcal{L}}\\
&=X_1(G-1)+(d_1-1)Y^2+AZ^2-Y(G-X_1)+\frac{d_2}{2d_1}\frac{d_2-1}{d_1+1}Z(G+X_1-2X_2)\\
&=X_1(G-d_1X_1-d_2X_2)+(d_1-1)Y^2+AZ^2-Y(G-X_1)+\frac{d_2}{2d_1}\frac{d_2-1}{d_1+1}Z(G+X_1-2X_2)
\end{split}
\end{equation}
and 
\begin{equation}
\label{eqn_RF deri2}
\begin{split}
&\left.\left(X_2-\frac{d_2-1}{d_1+1}Z\right)'\right|_{\mathcal{L}}\\
&=X_2(G-1)+(d_2-1)YZ-\frac{2d_1}{d_2}AZ^2-\frac{d_2-1}{d_1+1}Z(G+X_1-2X_2)\\
&=X_2(G-d_1X_1-d_2X_2)+(d_2-1)YZ-\frac{2d_1}{d_2}AZ^2-\frac{d_2-1}{d_1+1}Z(G+X_1-2X_2).
\end{split}
\end{equation}
Replace $(X_1,X_2)$ with the linear expressions in $(Y,Z)$, the above equations respectively become
$$
\left(A-\frac{d_2(d_2-1)^2}{4d_1^2(d_1+1)^2}(d_2+4d_1-d_1d_2)\right)Z^2,\quad 
\left(\frac{1}{2d_1}\frac{(d_2-1)^2}{(d_1+1)^2}(d_2+4d_1-d_1d_2)-\frac{2d_1}{d_2}A\right)Z^2,
$$
which vanish if \eqref{eqn_suffices to get a 1st integral} holds.
\end{proof}
\begin{table}[H]
\centering
\begin{tabular}{|l|l|l|l|l|}
\hline
$(d_1,d_2,A)$&Principal Orbit/Hypersurface & Holonomy & Source  \\
\hline
\hline
$\left(1,d,\frac{d(d-1)^2}{4}\right)$& $\mathbb{S}^1$-bundle associated to $\mathcal{O}(d+2)$ & K\"ahler&\cite{berard-bergery_sur_1982}\\ [1ex]
\hline
$\left(2,4,1\right)$& Twistor bundle over $\mathbb{HP}^1$ &  \multirow{3}{*}{$G_2$}    &\multirow{4}{*}{\cite{bryant_construction_1989}}\\ [1ex]
\cline{1-2}
$\left(3,3,\frac{1}{8}\right)$     & $\mathbb{S}^3\times \mathbb{S}^3$ && \\ [1ex]
\cline{1-3}
$\left(3,4,\frac{1}{4}\right)$& Hopf fibration over $\mathbb{HP}^1$& $\spin(7)$&\\ [1ex]
\hline
$(2,8,0)$&\multirow{3}{*}{Product space}&\multirow{3}{*}{Generic}& \multirow{3}{*}{\cite{dancer1999integrable}}\\
$(3,6,0)$&&&\\
$(5,5,0)$&&&\\
\hline
\end{tabular}
\caption{Examples that satisfy \eqref{eqn_suffices to get a 1st integral}}
\end{table}

\section{Computations of Rational functions}
\label{appendix:computation}
%\subsection{$\tilde{\omega}_2(\kappa_*)$}
%\begin{proposition}
%\label{prop_omegea2atkappa}
%For any $d_2\geq d_1\geq 2$ we have $\tilde{\omega}_2(\kappa_*)>0$.
%\end{proposition}
%\begin{proof}
%We have
%\begin{equation}
%\begin{split}
%\tilde{\omega}_2(\kappa_*)&=\frac{K_1-K_2\sqrt{(2d_1+d_2)^2+12d_2}}{K_0^3}\\
%&=\frac{1}{K_0^3}\frac{K_1^2-K_2^2((2d_1+d_2)^2+12d_2)}{K_1+K_2\sqrt{(2d_1+d_2)^2+12d_2}}
%\end{split}
%\end{equation}
%where
%\begin{equation}
%\begin{split}
%K_0&=5 d_2-2d_1 - \sqrt{(d_2+2d_1)^2+12d_2}\\
%K_1&=(4 d_1 - 4) d_2^4 + (24 d_1^2 + 60 d_1 + 56) d_2^3 + (48 d_1^3 + 144 d_1^2 - 20 d_1 - 88) d_2^2 \\
%&\quad + (32 d_1^4 + 16 d_1^3 + 16 d_1^2 - 8 d_1 - 32) d_2 - 16 d_1^3 - 16 d_1^2>0\\
%K_2&=(4 d_1 - 4) d_2^3 + (16 d_1^2 + 44 d_1 + 16) d_2^2 + (16 d_1^3 + 8 d_1^2 - 36 d_1 - 32) d_2 + 8 d_1^2 + 8 d_1>0.
%\end{split}
%\end{equation}
%We also have
%\begin{equation}
%\begin{split}
%&K_1^2-K_2^2((2d_1+d_2)^2+12d_2)\\
%&=64 d_2  (2 d_2 -2 d_1-1)^3((d_1 - 1) d_2^3 + (5 d_1^2 + 14 d_1 + 8) d_2^2 + (8 d_1^3 + 16 d_1^2 - 8 d_1 - 16) d_2 + 4 d_1^4 - 4 d_1^2).
%\end{split}
%\end{equation}
%If $d_1=d_2$, then both $K_1^2-K_2^2((2d_1+d_2)^2+12d_2)$ and $K_0$ are negative. If $d_1\leq d_2-1$, then both  $K_1^2-K_2^2((2d_1+d_2)^2+12d_2)$ and $K_0$ are positive. We conclude that $\tilde{\omega}_2(\kappa_*)>0$.
%\end{proof}

\begin{proposition}
\label{prop_PX+QX is negative}
The polynomial $P_X+Q_X$ is negative on $[0,1]$
\end{proposition}
\begin{proof}
For $k\in[0,1]$ we have 
\begin{equation}
\label{eqn_PX+QX is negative_1}
\begin{split}
-d_1^2(n-1)(P_X+Q_X)
%&=(d_1^2d_2^2 + d_1d_2^3 - 4d_1^2d_2 - 3d_1d_2^2 - d_2^3 + d_1d_2 + d_2^2)k^3 \\
%&\quad + (d_1^2d_2^2 - 2d_1^3 + 2d_1^2d_2 - d_1d_2^2 + 2d_1^2 + d_1d_2)k^2\\
%&\quad  + (-d_1^4 + d_1^3d_2 + 3d_1^3 - d_1^2d_2 - 2d_1^2)k + d_1^4-d_1^3\\
%&=(d_1^2d_2^2 + d_1d_2^3 - 4d_1^2d_2 - 3d_1d_2^2 - d_2^3 + d_1d_2 + d_2^2)k^3 \\
%&\quad + (d_1^2d_2^2  + 2d_1^2d_2 - d_1d_2^2  + d_1d_2)k^2\\
%&\quad  + ( d_1^3d_2 - d_1^2d_2)k+(2d_1^3  - 2d_1^2)(k-k^2) + d_1^3(d_1-1)(1-k)\\
&= (d_1^2d_2^2 + d_1d_2^3 - 2d_1^2d_2 - 2d_1d_2^2 - d_2^3 + d_1d_2 + d_2^2)k^3-(2d_1^2d_2 + d_1d_2^2)k^3\\
&\quad + (d_1^2d_2^2  + 2d_1^2d_2 - d_1d_2^2  + d_1d_2)k^2\\
&\quad  + ( d_1^3d_2 - d_1^2d_2)k+(2d_1^3  - 2d_1^2)(k-k^2) + d_1^3(d_1-1)(1-k)\\
&\geq (d_1^2d_2^2 + d_1d_2^3 - 2d_1^2d_2 - 2d_1d_2^2 - d_2^3 + d_1d_2 + d_2^2)k^3\\
&\quad + (d_1^2d_2^2   - 2d_1d_2^2  + d_1d_2)k^2\\
&\quad  + ( d_1^3d_2 - d_1^2d_2)k+(2d_1^3  - 2d_1^2)(k-k^2) + d_1^3(d_1-1)(1-k)\\
\end{split}
\end{equation}

The last four terms above are non-negative and do not vanish simultaneously for any $k\in [0,1]$. The coefficient of $k^3$ is also non-negative since
$$
d_1^2d_2^2 + d_1d_2^3 - 2d_1^2d_2 - 2d_1d_2^2 - d_2^3 + d_1d_2 + d_2^2=(d_2-2)(d_1 d_2(d_1+d_2)-d_2(d_2+1))+(d_1-2)d_2\geq 0.
$$
%If the coefficient of $k^3$ is non-negative, then the polynomial is positive on $[0,1]$. If the coefficient of $k^3$ is negative, we have 
%\begin{equation}
%\label{eqn_PX+QX is negative_2}
%\begin{split}
%-d_1^2(n-1)(P_X+Q_X)&> (d_1^2d_2^2 + d_1d_2^3 - 4d_1^2d_2 - 3d_1d_2^2 - d_2^3 + d_1d_2 + d_2^2)k^2 \\
%&\quad + (d_1^2d_2^2  + 2d_1^2d_2 - d_1d_2^2  + d_1d_2)k^2\\
%&= 2d_1d_2(d_1d_2-d_1-d_2)k^2 + (d_2^2-d_2)(d_1d_2-d_1-d_2)k^2+d_1d_2k^2.\\
%\end{split}
%\end{equation}
Therefore, given $d_1, d_2\geq 2$, the polynomial above is non-negative on $[0,1]$.
Hence, we have $P_X+Q_X<0$ on $[0,1]$. 
\end{proof}

\begin{proposition}
\label{prop_beta0 and beta_1}
For any $d_1, d_2 \geq 2$, the inequalities $\beta_0(k)>0$ and $\beta_1(k)>0$ hold for any $k\in[0,1]$.
\end{proposition}
\begin{proof}
Since
\begin{equation}
\label{eqn_beta 0>0}
\begin{split}
\beta_0(k)&=(d_1 d_2^2 - 2 d_1 d_2 - d_2^2 + d_2) k^2 + (2 d_1^2 d_2 - 2 d_1^2 - d_1 d_2 + 2 d_1) k + d_1^3 - d_1,\\
&\geq (d_1-1)(d_2-1)d_2 k^2 -d_1d_2k + (2 d_1^2 d_2 - 2 d_1^2 - d_1 d_2 + 2 d_1) k + d_1^3 - d_1\\
&= (d_1-1)(d_2-1)d_2 k^2  + 2d_1(d_1-1)(d_2-1)k + d_1^3 - d_1,\\
\end{split}
\end{equation}
we have $\beta_0>0$ on $[0,1]$.

Consider
\begin{equation}
\begin{split}
\beta_1(k)&=(2 d_1^2 d_2^2 + d_1 d_2^3 - 4 d_1^2 d_2 - 2 d_1 d_2^2 -  d_2^3 - 2 d_1 d_2 + 2 d_2) k^2 \\
&\quad + (4 d_1^3 d_2 + 2 d_1^2 d_2^2 - 4 d_1^3 - 2 d_1^2 d_2 - 3 d_1 d_2^2 + 4 d_1 d_2 + d_2^2 + 4 d_1 - 2 d_2) k\\
&\quad  + 2 d_1^4 + d_1^3 d_2 - 2 d_1^2 d_2 + 2 d_1^2 + d_1 d_2 - 4 d_1.\\
\end{split}
\end{equation}
Since 
\begin{equation}
\begin{split}
\beta_1(0)&=2 d_1^4 + d_1^3 d_2 - 2 d_1^2 d_2 + 2 d_1^2 + d_1 d_2 - 4 d_1\\
&\geq 2 d_1^4   + d_1 d_2\\
&>0
\end{split}
\end{equation}
and
\begin{equation}
\begin{split}
\frac{d\beta_1}{dk}(0)&=4 d_1^3 d_2 + 2 d_1^2 d_2^2 - 4 d_1^3 - 2 d_1^2 d_2 - 3 d_1 d_2^2 + 4 d_1 d_2 + d_2^2 + 4 d_1 - 2 d_2\\
&\geq 4 d_1^3 d_2 + 2 d_1^2 d_2^2 - 4 d_1^3 - 2 d_1^2 d_2 - 3 d_1 d_2^2 + 4 d_1 d_2  + 4 d_1 \\
&= (3 d_1^3 d_2 - 4 d_1^3)+ (d_1^3 d_2  - 2 d_1^2 d_2) + (2 d_1^2 d_2^2- 3 d_1 d_2^2) + 4 d_1 d_2 + 4 d_1 \\
& >0
\end{split}
\end{equation}
if $d_1,d_2\geq 2$, it is clear that $\beta_1>0$ on $[0,1]$ if the coefficient of $k^2$ is non-negative. Since 
$$
\beta_0(1)=(d_1-1)(d_1+d_2)(d_1+d_2-1)>0,
$$
we have $\beta_1>0$ on $[0,1]$ even if the coefficient of $k^2$ is negative.
\end{proof}

\begin{proposition}
\label{prop_min of beta0 over beta1}
$\min\limits_{k\in[0,1]} \left\{\frac{\beta_0}{\beta_1}\right\}=\frac{1}{n+d_1}.$
\end{proposition}
\begin{proof}
We compute
\begin{equation}
\begin{split}
\frac{d}{dk}\left(\frac{\beta_0}{\beta_1}\right)&=-\frac{(d_1-1)(d_2-2)}{\beta_1^2}\tilde{\beta},
\end{split}
\end{equation}
where
\begin{equation}
\begin{split}
\tilde{\beta}&=(2d_1^2d_2^2 + 2d_1d_2^3 - 2d_1d_2^2 - d_2^3 + d_2^2)k^2 + (4d_1^3d_2 + 4d_1^2d_2^2 - 6d_1^2d_2 - 4d_1d_2^2 + 2d_1d_2)k\\
&\quad + 2d_1^4 + 2d_1^3d_2 - 4d_1^3 - 3d_1^2d_2 + 2d_1^2 + d_1d_2.
\end{split}
\end{equation}
The factor $\tilde{\beta}$ is positive as each coefficient for $k^i$ is positive given $d_1, d_2\geq 2$.
Hence, we have $\min\limits_{k\in[0,1]} \left\{\frac{\beta_0}{\beta_1}\right\}=\frac{\beta_0}{\beta_1}(1)=\frac{1}{n+d_1}$. Note that if $(d_1,d_2)=(2,2)$, we have $\beta_0(k)=-2k^2 + 8k + 6=\frac{1}{6}\beta_1(k)$.
\end{proof}

\begin{proposition}
\label{prop_beta1}
For any $d_1, d_2\geq 2$, we have $\beta_2(k)>0$ for any $k\in [0,1]$.
\end{proposition}
\begin{proof}
We have
\begin{equation}
\begin{split}
\beta_2(k)
%&=(d_1d_2^3 - 2d_1d_2^2 - d_2^3 + d_2^2)k^3 + (2d_1^2d_2^2 - 4d_1^2d_2 - 2d_1d_2^2 + 2d_1d_2)k^2\\
%&\quad  + (d_1^3d_2 - 2d_1^3 + d_1^2d_2)k + 2d_1^3\\
&=(d_1-1)(d_2 -1)d_2^2k^3 +d_1d_2^2 (k^2-k^3)+ (2d_1^2d_2^2 - 4d_1^2d_2 - 3d_1d_2^2 + 2d_1d_2)k^2\\
&\quad  + (d_1^3d_2 + d_1^2d_2)k + 2d_1^3(1-k)\\
&>   (2d_1^2d_2^2 - 4d_1^2d_2 - 3d_1d_2^2 + 2d_1d_2)k^2+ (d_1^3d_2 + d_1^2d_2)k\\
&\geq   (- 4d_1^2d_2 + 2d_1d_2)k^2+ (d_1^3d_2 + d_1^2d_2)k\\
&\geq   (d_1^3d_2 +2d_1d_2 -3 d_1^2d_2)k\\
&\geq   0.
\end{split}
\end{equation}
We conclude that $\beta_2> 0$ on $[0,1]$.
\end{proof}

\begin{proposition}
\label{prop_formula_of_Theta}
The coefficients for $\Theta_{d_1,d_2}(A,k)$ is given by 
\begin{equation}
\theta_2=4d_1^4(d_1+1)^2\theta_{2a}\theta_{2b},\quad \theta_1=\theta_{1a}\theta_{1b},\quad \theta_0=\theta_{0a}\theta_{0b},
\end{equation}
where
\begin{equation}
\label{eqn_coefficient theta}
\begin{split}
\theta_{2a}(k)&=(d_1 d_2^3 - 2 d_1 d_2^2 - d_2^3 - 2 d_1 d_2 + d_2^2) k^2 \\
&\quad + (2 d_1^2 d_2^2 - 2 d_1^2 d_2 - 2 d_1 d_2^2 - 4 d_1^2 + 4 d_1 d_2) k + d_1^3 d_2 - d_1^2 d_2 + 4 d_1^2,\\
\theta_{2b}(k)&=(2d_1^2d_2^2 - d_1d_2^3 + d_2^3 - d_2^2)k^3\\
&\quad  + (4d_1^3d_2 - 4d_1^2d_2^2 - 2d_1^2d_2 + 4d_1d_2^2 - 2d_1d_2)k^2\\
&\quad  + (2d_1^4 - 5d_1^3d_2 - 2d_1^3 + 5d_1^2d_2)k - 2d_1^4 + 2d_1^3,\\
\theta_{1a}(k)&=-4d_1^2d_2(d_2-1)^2(d_1+1),\\
\theta_{1b}(k)&=(d_1^3 d_2^5 - d_1^2 d_2^5 - 4 d_1^3 d_2^3 + 3 d_1^2 d_2^4 - d_1 d_2^4 + 2 d_1 d_2^3 - d_2^4) k^5\\
&\quad  + (4 d_1^4 d_2^4 + 4 d_1^4 d_2^3 - 6 d_1^3 d_2^4 - 16 d_1^4 d_2^2 + 18 d_1^3 d_2^3 + 2 d_1^2 d_2^4\\
&\quad \quad  + 4 d_1^3 d_2^2 - 10 d_1^2 d_2^3 + 8 d_1^2 d_2^2 - 6 d_1 d_2^3) k^4\\
&\quad  + (6 d_1^5 d_2^3 + 12 d_1^5 d_2^2 - 14 d_1^4 d_2^3 - 20 d_1^5 d_2 + 35 d_1^4 d_2^2 + 7 d_1^3 d_2^3 \\
&\quad \quad  + 12 d_1^4 d_2 - 29 d_1^3 d_2^2 + d_1^2 d_2^3 + 8 d_1^3 d_2 - 12 d_1^2 d_2^2) k^3\\
&\quad  + (4 d_1^6 d_2^2 + 12 d_1^6 d_2 - 16 d_1^5 d_2^2 - 8 d_1^6 + 28 d_1^5 d_2 + 8 d_1^4 d_2^2\\
&\quad \quad  + 8 d_1^5 - 32 d_1^4 d_2 + 4 d_1^3 d_2^2 - 8 d_1^3 d_2) k^2 \\
&\quad + (d_1^7 d_2 + 4 d_1^7 - 9 d_1^6 d_2 + 8 d_1^6 + 3 d_1^5 d_2 - 12 d_1^5 + 5 d_1^4 d_2) k\\
&\quad  - 2 d_1^7 + 2 d_1^5,\\
\theta_{0a}(k)&=k d_2^2 (d_2 - 1)^4 (2 d_1+d_2k) ((2 d_1d_2-2d_1+d_2 )k + 2 d_1^2 + 2 d_1),\\
\theta_{0b}(k)&=(2d_1^2-1)d_2^2k^2+(4d_1^3-2d_1^2-2d_1)d_2k+2d_1^3(d_1-1).
\end{split}
\end{equation}
\end{proposition}
\begin{proof}
We use Maple to obtain the formulas.
\end{proof}

\begin{proposition}
\label{prop_coefficients for zeta}
For any $d_1, d_2\geq 2$, we have $\theta_2<0$ and $\theta_0\geq 0$ on $[0,1]$. The function $\theta_0$ only vanishes at $k=0$.
\end{proposition}
\begin{proof}
From observation it is clear that $\theta_0=\theta_{0a}\theta_{0b}\geq 0$ on $[0,1]$ and only vanishes at $k=0$. For $\theta_2$ we have 
\begin{equation}
\begin{split}
\theta_{2a}
&=(d_1 d_2^3 - 2 d_1 d_2^2 - d_2^3 - 2 d_1 d_2 + d_2^2) k^2 \\
&\quad + (2 d_1^2 d_2^2 - 2 d_1^2 d_2 - 2 d_1 d_2^2 - 4 d_1^2 + 4 d_1 d_2) k + d_1^3 d_2 - d_1^2 d_2 + 4 d_1^2\\
&=(d_1 d_2^3 - 2 d_1 d_2^2 - d_2^3 - 2 d_1 d_2 + d_2^2) k^2 \\
&\quad + (2 d_1^2 d_2^2 - 2 d_1^2 d_2 - 2 d_1 d_2^2 + 4 d_1 d_2) k + d_1^3 d_2 - d_1^2 d_2 + 4 d_1^2(1-k)\\
&\geq (d_1 d_2^3 - d_2^3 - 2 d_1 d_2 + d_2^2) k^2 \\
&\quad + (2 d_1^2 d_2^2 - 2 d_1^2 d_2 - 4 d_1 d_2^2 + 4 d_1 d_2) k  + 4 d_1^2(1-k)+ d_1^2 d_2(d_1-1).
\end{split}
\end{equation}
Given $d_1, d_2\geq 2$, each term above is non-negative on $[0,1]$. In particular, the last term is positive. Hence $\theta_{2a}>0$ on $[0,1]$.

For the other factor $\theta_{2b}$, we have
\begin{equation}
\begin{split}
\theta_{2b}
%&=(2d_1^2d_2^2 - d_1d_2^3 + d_2^3 - d_2^2)k^3\\
%&\quad  + (4d_1^3d_2 - 4d_1^2d_2^2 - 2d_1^2d_2 + 4d_1d_2^2 - 2d_1d_2)k^2\\
%&\quad  + (2d_1^4 - 5d_1^3d_2 - 2d_1^3 + 5d_1^2d_2)k - 2d_1^4 + 2d_1^3\\
&=-(d_1d_2^3 - d_2^3 + d_2^2)k^3- 2d_1^2d_2^2 (k^2-k^3)\\
&\quad  - (d_1^3d_2  -3 d_1^2d_2+ 2d_1d_2)k^2  -(2d_1^2d_2^2-4d_1d_2^2 )k^2\\
&\quad  - ( 5d_1^3d_2  - 5d_1^2d_2)(k-k^2) - (2d_1^4 - 2d_1^3)(1-k).
\end{split}
\end{equation}
Given $d_1, d_2\geq 2$, all terms above are non-positive and do not vanish simultaneously for any $k\in[0,1]$. We have $\theta_{2b}<0$ on $[0,1]$.
The proof is complete.
\end{proof}

\section{Examples}
\label{appendix:PEexamples}
Recall in Section~\ref{sec_Global Analysis}, we have $\chi_{d_1,d_2}=\left\{\frac{d_2(d_2-1)^2}{d_1^2(d_1d_2-d_2+4)},A_1,A_2\right\}$. The constant $A_1$ is the minimum of the rational function $F$ in Proposition \ref{prop_mu2<nu1<mu1}, which can be explicitly computed with a fixed $(d_1,d_2)$. On the other hand, the constant $A_2$ is the only positive root of the quadratic function $\Theta_{d_1,d_2}(A,k)$ as in Proposition \ref{prop_formula_of_Theta}. Instead of computing $A_2$, it suffices to check $\Theta_{d_1,d_2}(A,k)>0$ for any $k\in[0,1]$. Following the guideline described above, we present the following two examples.

\begin{example}
Consider the twistor bundle over $\mathbb{HP}^1$, whose group triple is
$$(\mathsf{K},\mathsf{H},\mathsf{G})=(Sp(1)U(1)U(1),Sp(1)Sp(1),Sp(2)Sp(1)).$$
The associated structural triple is $(2,4,1)$.

For $(d_1,d_2)=(2,4)$, we have 
$$
F=\frac{9 \left(16 k^{3}+16 k^{2}+32 k+16\right)}{\left(4 k+4\right) \left(48 k+32\right)}.
$$
Since 
$$
F'=\frac{9 \left(k^{2}+3 k+1\right) \left(3 k^{2}+k-1\right)}{4 \left(k+1\right)^{2} \left(3 k+2\right)^{2}}
$$
vanishes at $k_*=-\frac{1}{6}+\frac{\sqrt{13}}{6}\approx 0.434$
we have $A_1=\min\limits_{k\in[0,1]} F=F(k_*)=\frac{19 \sqrt{13}+47}{56+16\sqrt{13}}\approx 1.016>1$

On the other hand, we have 
\footnotesize
\begin{equation}
\begin{split}
\Theta_{2,4}\left(1,k\right)&=1327104 k^5 + 1658880 k^4 + 663552 k^3 + 221184 k^2 + 147456 k + 36864>0.
\end{split}
\end{equation}
\normalsize
Hence $A_2>1$. With $\frac{d_2(d_2-1)^2}{d_1^2(d_1d_2-d_2+4)}=\frac{9}{8}>1$, we know that $1<\chi_{2,4}$. The Einstein metric on $F^2$ predicted in \cite{gibbons_einstein_1990,bohm_inhomogeneous_1998,dancer_cohomogeneity_2013} is recovered.
\end{example}

\begin{example}
Consider the quaternionic Hopf fibration over $\mathbb{HP}^1$, whose group triple is
$$(\mathsf{K},\mathsf{H},\mathsf{G})=(Sp(1)\triangle Sp(1),Sp(1)Sp(1)Sp(1),Sp(2)Sp(1)).$$
The associated structural triple is $\left(3,4,\frac{1}{4}\right)$.

For $(d_1,d_2)=(3,4)$, we have 
$$
F=\frac{4 \left(48 k^{3}+72 k^{2}+90 k +54\right)}{\left(4 k +6\right) \left(24 k^{2}+132 k +108\right)}.
$$
Since 
$$
F'=\frac{2 \left(2 k^{2}+6 k+3\right) \left(44 k^{2}+42 k-9\right)}{\left(2 k+3\right)^{2} \left(2 k+9\right)^{2} \left(k+1\right)^{2}}
$$
vanishes at $k_*=-\frac{21}{44}+\frac{3 \sqrt{93}}{44}\approx 0.180$
we have $A_1=\min\limits_{k\in[0,1]} F_1=F_1(k_*)=\frac{503 \sqrt{93}+2331}{10785+1159 \sqrt{93}}\approx 0.327>\frac{1}{4}$

On the other hand, we have 
\footnotesize
\begin{equation}
\begin{split}
\Theta_{3,4}\left(\frac{1}{4},k\right)&=9082368 k^{5}+29673216 k^{4}+38351232 k^{3}+24774336 k^{2}+8398080 k+1259712>0.
\end{split}
\end{equation}
\normalsize
Hence $A_2>\frac{1}{4}$. With $\frac{d_2(d_2-1)^2}{d_1^2(d_1d_2-d_2+4)}=\frac{1}{3}$, we know that $\frac{1}{4}<\chi_{3,4}$. The Einstein metric on $\mathbb{HP}^2 \sharp\overline{\mathbb{HP}}^{2}$ in \cite{bohm_inhomogeneous_1998} is recovered.
\end{example}

\textbf{Acknowledgement.}
The author would like to thank the organizers of the conference ``Einstein Spaces and Special Geometry'' at the Institut Mittag-Leffler in July 2023, where part of this work was presented. The author thanks Christoph B\"ohm, Andrew Dancer, Lorenzo Foscolo, and Claude LeBrun for insightful discussions during the conference. Special thanks to Christoph B\"ohm for helpful suggestions that improved the exposition and overall structure of the paper. The author also thanks McKenzie Wang and Xiping Zhu for valuable comments on an early draft of this paper.

\bibliography{BIB}
\bibliographystyle{alpha}
\end{document}